\documentclass[ks]{imsart}
\usepackage{amsmath,amssymb,amsthm,mathtools}
\usepackage[margin=1in]{geometry}
\usepackage{url}

\usepackage{url}
\usepackage{hyperref}

\usepackage{subcaption}
\usepackage{enumitem}

\theoremstyle{plain}

\newtheorem{theorem}{Theorem}[section]
\newtheorem{lemma}[theorem]{Lemma}

\newtheorem{corollary}[theorem]{Corollary}

\theoremstyle{definition}

\theoremstyle{remark}

\usepackage{tikz}
\usepackage{pgfplots}
\usepackage{xcolor}
\usetikzlibrary{calc}

\usetikzlibrary{matrix,positioning,quotes,fit}
\usetikzlibrary{backgrounds}
\usetikzlibrary {arrows.meta}
\usetikzlibrary{decorations.pathmorphing}
\pgfplotsset{compat=1.18}

\tikzset{standard/.style={matrix of nodes,left delimiter={(},right delimiter={)},inner sep=0pt,nodes={inner sep=0.3em}}}
\tikzset{submatrix/.style = {rectangle, rounded corners,
  fill=yellow, draw, inner sep=0pt}}

\usepackage{relsize}
\usepackage{scalefnt}

\tikzset{fontscale/.style = {font=\relsize{#1}}}

\usepackage{tikz-3dplot}    
\usetikzlibrary{intersections,fadings,decorations.pathreplacing} 

\makeatletter
\newcommand{\tikzAngleOfLine}{\tikz@AngleOfLine}
\def\tikz@AngleOfLine(#1)(#2)#3{%
  \pgfmathanglebetweenpoints{\pgfpointanchor{#1}{center}}{\pgfpointanchor{#2}{center}}
  \pgfmathsetmacro{#3}{\pgfmathresult}%
}
\makeatother

\tikzset{
    vector/.style = {thick,> = stealth',},
    vector_thick/.style = {very thick,> = stealth',},
    vector_thin/.style = {thin,> = stealth',},        
    axis/.style = {very thin,> = stealth',},
    axis_thick/.style = {thick,> = stealth',},    
}

\tikzstyle{dim}    = [latex-latex]

\begin{document}
\begin{frontmatter}
\title{Constructive Proofs of Generalized Boole--Fr\'{e}chet Bounds: A Dynamic Programming Approach}

\begin{aug}
\author{\fnms{Kizito}~\snm{Salako}\ead[label=e1]{k.o.salako@citystgeorges.ac.uk}\orcid{0000-0003-0394-7833}}
\address{The Centre for Software Reliability, Department of Computer Science, \\City St. George's, University of London, \\Northampton square, EC1V 0HB, The United Kingdom\\\printead{e1}}

\end{aug}

\begin{abstract}
Extensions of the Boole--Fr\'{e}chet inequalities give sharp bounds for the probabilities of compound events, particularly when only the probabilities of atomic events (that make up the compound events) are known. We present a constructive approach to obtaining generalized Boole--Fr\'{e}chet bounds using dynamic programming. 
\end{abstract}



\end{frontmatter}

\maketitle

\section{Introduction. }
\label{sec_intro}
The Boole--Fr\'{e}chet inequalities~---~introduced by Boole~\cite{Boole_2009} and Fréchet~\cite{Frechet_1935}~---~provide sharp (best-possible) bounds for the probabilities of finite unions and intersections given only the events’ marginal probabilities. Using linear programming, Hailperin \cite{Hailperin_1965} generalized these inequalities significantly, to bound probabilities of \emph{any} Boolean function of finitely many atomic events.  
Several extensions of the original Boole--Fr\'{e}chet inequalities have been studied. Examples include early works by Bonferroni \cite{Bonferroni_1937}, Chung et al. \cite{ChungErdos_1952}, Gallot \cite{Gallot_1966}, Dawson et al. \cite{1967_Dawson}, and Kounias \cite{Kounias_1968}, which bound probabilities of conjunctions and disjunctions of atomic events, where (sums of) the probabilities of these atomic events, and (sums of) the probabilities of pairs of these atomic events, are known. Refinements to these bounds have been derived; e.g. see Kounias et al. \cite{Kounias_1976}, while Hunter uses a graph-theoretic algorithm to compute the bounds \cite{Hunter_1976}. These bounds, and more, have been derived under various constraints~---~e.g. knowledge of more \emph{binomial moments}, or statistical independence between atomic events~---~see \cite{Galambos_1996,Prekopa_2005,Bukszar_2012,Boros_2014}. Puccetti et al.'s rearrangement algorithm numerically approximates generalized Boole--Fr\'{e}chet bounds \cite{PuccettiRueschendorf2012}.

We develop a dynamic programming framework for constructive proofs of generalized Boole-Fr\'{e}chet bounds. An equivalence with Hailperin’s linear programming formulation shows the bounds are fixed--point values of related Bellman operators. In particular, for ``$k$--out--of--$n$'' events, we obtain closed--form expressions for the bounds, explicit characterizations of the extremal distributions that attain the bounds, and a stopping rule that can be used in efficient search algorithms for the bounds. The outline of the paper is: Section~\ref{sec_prelim} contains preliminary definitions and concepts, the constructive approach is detailed in Section~\ref{sec_bounds}, with final remarks and discussion in Section~\ref{sec_discussion}.    


\section{Preliminaries. }
\label{sec_prelim}
Consider events $A_1,\ldots,A_n$ and probabilities $0<p_1<\ldots<p_n<1$, such that $\mathbb{P}(A_i)=p_i$ for some unknown probability distribution $\mathbb{P}$. Let $\boldsymbol I$ be the set of all $n$-length binary words. For $I\in\boldsymbol I$, $A_{I}$ is the intersection of $n$ events from $\{A_1, A_1^c,\ldots,A_n, A_n^c \}$ that satisfies $A_I\subseteq A_i$ if the $i$-th component of $I$ is $1$, otherwise $A_I\subseteq A_i^c$. For example, with events $A_1, A_2, A_3$, we have $A_{101}:=A_1\cap A_2^c\cap A_3$. We write $p_I:=\mathbb{P}(A_I)$, so $p_i=\sum_{I\in{\boldsymbol I}: A_I\subseteq A_i}p_I$ for $i=1,\ldots,n$. 

Let $|I|$ be the sum of all the ``$1$''s in $I$ --- i.e. the so-called \emph{Hamming weight}. For $|I|>0$, the probability $p_I$ can be viewed as a stack of $|I|$ overlapping layers, with each layer of length $p_I$; indeed, \emph{w.l.o.g.}, one may visualize this using partitions of the unit interval (e.g. Fig.~\ref{fig_overlaps_intro}). More specifically, if a distribution assigns the $p_i$ probabilities to the $A_i$ events, its $p_I$ probabilities partition the unit interval into sub-intervals. These sub-intervals, in turn, induce stacks of overlapping horizontal layers as illustrated in Fig.~\ref{fig_overlaps_intro}. The number of layers, $|I|$, in the stack representing $p_I$, is the number of ``$A$'' events indicated by $I$. 
The event $\cap_{i=1}^{n}A_i^c$ indicated by the binary word $I=0\ldots 0$ does not have its probability $p_{0\ldots 0}$ represented by a stack of ``$A$'' layers: no ``$A$'' events contribute to this compound event.

\begin{figure}[t!]
    \centering
        \begin{tikzpicture}
            \begin{scope}
                \coordinate (r1lu) at (0.3,-1.85);
                \coordinate (r1ld) at (0.3,-1.75);
                \coordinate (r1rd) at (2.5,-1.75);
                \coordinate (r1ru) at (2.5,-1.85);
                \fill[fill=gray!70] (r1lu) -- (r1ld) -- (r1rd) -- (r1ru) -- cycle;

                \coordinate (r2lu) at (1.8,-0.85);
                \coordinate (r2ld) at (1.8,-0.75);
                \coordinate (r2rd) at (5,-0.75);
                \coordinate (r2ru) at (5,-0.85);
                \fill[fill=gray!70] (r2lu) -- (r2ld) -- (r2rd) -- (r2ru) -- cycle;

                \coordinate (r3lu) at (0.3,-1.35);
                \coordinate (r3ld) at (0.3,-1.25);
                \coordinate (r3rd) at (5,-1.25);
                \coordinate (r3ru) at (5,-1.35);
                \fill[fill=gray!70] (r3lu) -- (r3ld) -- (r3rd) -- (r3ru) -- cycle;  

                \draw[line width=0.07cm,white] (1.8,-0.7) -- (1.8,-2.5);
                
                \draw[line width=0.07cm,white] (2.5,-0.7) -- (2.5,-2.5);





                    
                \draw[-{Latex[length=3mm]},line width=0.06cm] (-1,-2.25) -- (5.85,-2.25); 
                \node[anchor=north,scale=1] (x_axis_label) at (5,-2.65) {$\pmb{1}$}; 

                \node[anchor=north east,scale=1] (x_axis_label) at (-0.55,-2.62) {$\pmb{0}$}; 

                \draw[-{Latex[length=3mm]},line width=0.06cm] (-1,-2.25) -- (5.85,-2.25); 

                \coordinate (p111_l) at (1.8,-3);
                \coordinate (p111_r) at (2.55,-3);
                \draw[|-|,line width=0.06cm] ($(p111_l)!5mm!90:(p111_r)$)--($(p111_r)!5mm!-90:(p111_l)$);
                \node[anchor=west,scale=1] (p111_lab) at (1.8,-2.9) {$p_{111}$}; 

                \coordinate (p101_l) at (0.3,-3);
                \coordinate (p101_r) at (1.75,-3);
                \draw[|-|,line width=0.06cm] ($(p101_l)!5mm!90:(p101_r)$)--($(p101_r)!5mm!-90:(p101_l)$);
                \node[anchor=west,scale=1] (p101_lab) at (0.65,-2.9) {$p_{101}$}; 

                \coordinate (p011_l) at (2.6,-3);
                \coordinate (p011_r) at (5,-3);
                \draw[|-|,line width=0.06cm] ($(p011_l)!5mm!90:(p011_r)$)--($(p011_r)!5mm!-90:(p011_l)$);
                \node[anchor=west,scale=1] (p011_lab) at (3.45,-2.9) {$p_{011}$}; 

                \coordinate (p000_l) at (-0.85,-3);
                \coordinate (p000_r) at (0.26,-3);
                \draw[|-|,line width=0.06cm] ($(p000_l)!5mm!90:(p000_r)$)--($(p000_r)!5mm!-90:(p000_l)$);
                \node[anchor=west,scale=1] (p000_lab) at (-0.65,-2.9) {$p_{000}$}; 
                
                \node[anchor=west,scale=1] (p1_lab) at (-1,-1.8) {$p_{1}$}; 
                \node[anchor=west,scale=1] (p2_lab) at (-1,-0.8) {$p_{2}$}; 
                \node[anchor=west,scale=1] (p3_lab) at (-1,-1.3) {$p_{3}$}; 

            \end{scope}    
        \end{tikzpicture}
    \caption{
    }
    \label{fig_overlaps_intro}
\end{figure}
\section{Results. }
\label{sec_bounds}

\subsection{Construction of Sharp Bounds on ``$\boldsymbol k$--out--of--$\boldsymbol n$'' Probabilities.}
\label{subsec_constructedkoutofnbounds}

\begin{theorem}
\label{thrm_sup}
Let $k\in\{1,\ldots,n\}$. Then,
\begin{align}
\label{eqn_optprob}
&\sup\limits_{\mathbb{P}} \mathbb{P}(\text{at least }k\text{--out--of--}n\text{ events}) \nonumber\\
\text{s.t.}\,&\,\, \mathbb{P}(A_1)=p_1,\ldots,\mathbb{P}(A_n)=p_n
\end{align}
has the solution 
\begin{align}
\label{eqn_soln}
\min\!\left\{\sum_{i=1}^{n-r^\ast}\frac{p_i}{k-r^\ast},\,1\right\}
\end{align}
where 
$$r^\ast:=\left\{\begin{array}{ll}\max\!\left\{r\in\{1,\ldots,k-1\}\,\,\middle|\,\, p_{n-r+1} \geqslant \sum\limits_{i=1}^{n-r+1}\left(\frac{p_i}{k-r+1}\right)\right\}\!,&\text{if max exists} \\ 0,&\text{otherwise}\end{array}\right.$$
\end{theorem}

\begin{proof}
In 4 steps, construct a probability distribution that attains the supremum \eqref{eqn_soln}:
\begin{enumerate}
\item There exist feasible distributions that satisfy the constraints in \eqref{eqn_optprob} (e.g. stack all $n$ of the ``$A$'' events' layers over the unit interval, \emph{\`{a} la} Fig.\ \ref{fig_overlaps_intro}). Any feasible distribution $\mathbb{P}$ can be used to construct a feasible distribution with objective function value that is no smaller. The constructed distribution can assign non-zero probability only to compound events: \textbf{i)} involving at least $k$ of the ``$A$''s, or \textbf{ii)} only some $A_{n-k+2},\ldots,A_{n}$, or \textbf{iii)} the event $\cap_{i=1}^{n}A_i^c$. Restrict \eqref{eqn_optprob} to the set $\mathcal D$ of these constructed distributions;
\item Any ${\mathbb{P}}\in\mathcal D$ can be used to construct a feasible distribution with objective function value that is no smaller. The constructed distribution can assign non-zero probability only to compound events: \textbf{i)} involving exactly $k$ many ``$A$''s, or \textbf{ii)} only some $A_{n-k+2},\ldots,A_n$, or \textbf{iii)} more than $k$ many ``$A$''s (only when $\cap_{i=1}^{n}A_i^c$ is a null event), or \textbf{iv)} the event $\cap_{i=1}^{n}A_i^c$. Restrict \eqref{eqn_optprob} to the set ${\mathcal D}^{'}$ of these distributions;
\item Any ${{\mathbb{P}}\in\mathcal D}^{'}$ can be used to construct a feasible distribution with objective function value that is no smaller. The constructed distribution can assign non-zero probability only to compound events: \textbf{i)} involving exactly $k$ many ``$A$''s, or \textbf{ii)} more than $k$ many ``$A$''s (only when $\cap_{i=1}^{n}A_i^c$ is a null event), or \textbf{iii)} the event $\cap_{i=r}^{n}A_{i}$ for some $(n-k+2)\leqslant r\leqslant n$ (where $\cap_{i=r}^{n}A_{i}$ forms part of all compound events with exactly $k$ many ``$A$''s), or \textbf{iv)} the event $\cap_{i=1}^{n}A_i^c$. Restrict \eqref{eqn_optprob} to the set ${\mathcal D}^{''}$ of these distributions;
\item Prove the supremum is \eqref{eqn_soln} by using ${\mathcal D}^{''}$.
\end{enumerate}
What follows is an outline of these 4 steps in more detail.

\paragraph{{\textbf{step 1:}}}
\label{subsec_proof_stage1}
Let $\mathbb{P}$ be any distribution that satisfies the constraints in \eqref{eqn_optprob}. If $\mathbb{P}$ assigns non-zero probabilities to compound events involving fewer than $k$ of the ``$A$''s --- including $A_i$ for some $i< (n-k+2)$ --- then $\mathbb{P}$ can be transformed without reducing the objective function value: create a new stack by overlapping the layers from all of these ``fewer than $k$ layers'' stacks\footnote{An event that contributes layers to more than one of these stacks can only contribute a single layer to the new stack. After forming the new stack, any ``leftover'' layers for this event remain at their original locations.}. If this new stack does not consist of $k$, or more, of the ``$A$''s (including $A_i$), then there must be some $A_j$ that is not involved in this new stack (where $j\geqslant n-k+2$); otherwise, there is a contradiction: this new stack must include layers for $A_{n-k+2},A_{n-k+3},\ldots,A_n,A_i$, which is $k$ of the ``$A$''s after all! This ``$A_j$'' layer must be part of other stacks with $k$, or more, layers. Since $p_j>p_i$, there must be a stack containing an $A_j$ layer and no $A_i$ layer. So, swap the $A_i$ layers --- in stacks with at most ($k-1$) layers --- with $A_j$ layers --- in stacks with at least $k$ layers that do not include $A_i$. Doing this ensures that any stack of at most $k-1$ layers will only involve ``$A$'' events such as $A_{n-k+2},\ldots,A_n$. Consequently, this transformation restricts the set of feasible distributions, to those distributions that assign non-zero probability only to compound events: \textbf{i)} involving at least $k$ of the ``$A$''s, or \textbf{ii)} involving only some $A_{n-k+2},\ldots,A_n$, or \textbf{iii)} the event $\cap_{i=1}^{n}A_i^c$. Denote this restricted set $\mathcal D$.

\begin{figure}[t!]
\captionsetup[figure]{format=hang}
    \begin{subfigure}[]{1\linewidth}
        \centering
	\begin{tikzpicture}
            \begin{scope}
                \coordinate (r1lu) at (0.3,-1.85);
                \coordinate (r1ld) at (0.3,-1.75);
                \coordinate (r1rd) at (2.5,-1.75);
                \coordinate (r1ru) at (2.5,-1.85);
                \fill[fill=gray!70] (r1lu) -- (r1ld) -- (r1rd) -- (r1ru) -- cycle;

                \coordinate (r2lu) at (0.3,-0.85);
                \coordinate (r2ld) at (0.3,-0.75);
                \coordinate (r2rd) at (2.5,-0.75);
                \coordinate (r2ru) at (2.5,-0.85);
                \fill[fill=gray!70] (r2lu) -- (r2ld) -- (r2rd) -- (r2ru) -- cycle;

                \coordinate (r3lu) at (0.3,-1.35);
                \coordinate (r3ld) at (0.3,-1.25);
                \coordinate (r3rd) at (2.5,-1.25);
                \coordinate (r3ru) at (2.5,-1.35);
                \fill[fill=gray!70] (r3lu) -- (r3ld) -- (r3rd) -- (r3ru) -- cycle;  

                




                    
                \draw[-{Latex[length=3mm]},line width=0.06cm] (-1,-2.25) -- (5.85,-2.25); 
                \node[anchor=north,scale=1] (x_axis_label) at (5,-2.65) {$\pmb{1}$}; 

                \node[anchor=north east,scale=1] (x_axis_label) at (-0.55,-2.62) {$\pmb{0}$}; 

                \draw[-{Latex[length=3mm]},line width=0.06cm] (-1,-2.25) -- (5.85,-2.25); 

                \coordinate (p111_l) at (0.3,-3);
                \coordinate (p111_r) at (2.55,-3);
                \draw[|-|,line width=0.06cm] ($(p111_l)!5mm!90:(p111_r)$)--($(p111_r)!5mm!-90:(p111_l)$);
                \node[anchor=west,scale=1] (p111_lab) at (1.1,-2.9) {$p_{111}$}; 


                \coordinate (p000_l) at (2.6,-3);
                \coordinate (p000_r) at (5,-3);
                \draw[|-|,line width=0.06cm] ($(p011_l)!5mm!90:(p011_r)$)--($(p011_r)!5mm!-90:(p011_l)$);

                \coordinate (p000_l) at (-0.85,-3);
                \coordinate (p000_r) at (0.26,-3);
                \draw[|-|,line width=0.06cm] ($(p000_l)!5mm!90:(p000_r)$)--($(p000_r)!5mm!-90:(p000_l)$);

                \node[anchor=west,scale=1] (G6) at (1.35,-3.5) {$p_{000}$};
                \node[scale=1.5] (G7) at (-0.15,-3.3) {$\pmb{\underbrace{}{}}$};    
                \draw[line width = {0.05cm}] (-0.2,-3.5) .. controls (-0.2,-3.9) .. (G6);
                \node[scale=3] (G8) at (3.95,-3) {$\pmb{\underbrace{}{}}$};
                \draw[line width = {0.05cm}] (3.85,-3.4) .. controls (4,-3.9) .. (G6);
    

                \node[anchor=west,scale=1] (p1_lab) at (-1,-1.8) {$p_{1}$}; 
                \node[anchor=west,scale=1] (p2_lab) at (-1,-0.8) {$p_{2}$}; 
                \node[anchor=west,scale=1] (p3_lab) at (-1,-1.3) {$p_{3}$}; 

            \end{scope}    
        \end{tikzpicture}
    \caption{\textit{This stack of 3 layers represents $p_{111}$. Stacks with 2 layers will be created from this.}
    }
	\label{fig_splittingoverlappinglayers_1}
\end{subfigure}
\par\bigskip
\begin{subfigure}[]{1.0\linewidth}
    \centering
	\begin{tikzpicture}
            \begin{scope}
                \coordinate (r1lu) at (0.3,-1.85);
                \coordinate (r1ld) at (0.3,-1.75);
                \coordinate (r1rd) at (2.5,-1.75);
                \coordinate (r1ru) at (2.5,-1.85);
                \fill[fill=gray!70] (r1lu) -- (r1ld) -- (r1rd) -- (r1ru) -- cycle;

                \coordinate (r2lu) at (0.3,-0.85);
                \coordinate (r2ld) at (0.3,-0.75);
                \coordinate (r2rd) at (2.5,-0.75);
                \coordinate (r2ru) at (2.5,-0.85);
                \fill[fill=gray!70] (r2lu) -- (r2ld) -- (r2rd) -- (r2ru) -- cycle;

                \coordinate (r3lu) at (0.3,-1.35);
                \coordinate (r3ld) at (0.3,-1.25);
                \coordinate (r3rd) at (2.5,-1.25);
                \coordinate (r3ru) at (2.5,-1.35);
                \fill[fill=gray!70] (r3lu) -- (r3ld) -- (r3rd) -- (r3ru) -- cycle;  

                \draw[line width=0.07cm,white] (1.4,-0.7) -- (1.4,-2.5);
                
                \draw[line width=0.07cm,white] (2.5,-0.7) -- (2.5,-2.5);





                    
                \draw[-{Latex[length=3mm]},line width=0.06cm] (-1,-2.25) -- (5.85,-2.25); 
                \node[anchor=north,scale=1] (x_axis_label) at (5,-2.65) {$\pmb{1}$}; 

                \node[anchor=north east,scale=1] (x_axis_label) at (-0.55,-2.62) {$\pmb{0}$}; 

                \draw[-{Latex[length=3mm]},line width=0.06cm] (-1,-2.25) -- (5.85,-2.25); 

                \coordinate (p111_l) at (0.3,-3);
                \coordinate (p111_r) at (2.55,-3);
                \draw[|-|,line width=0.06cm] ($(p111_l)!5mm!90:(p111_r)$)--($(p111_r)!5mm!-90:(p111_l)$);
                \node[anchor=west,scale=1] (p111_lab) at (1.1,-2.9) {$p_{111}$}; 


                \coordinate (p000_l) at (2.6,-3);
                \coordinate (p000_r) at (5,-3);
                \draw[|-|,line width=0.06cm] ($(p011_l)!5mm!90:(p011_r)$)--($(p011_r)!5mm!-90:(p011_l)$);

                \coordinate (p000_l) at (-0.85,-3);
                \coordinate (p000_r) at (0.26,-3);
                \draw[|-|,line width=0.06cm] ($(p000_l)!5mm!90:(p000_r)$)--($(p000_r)!5mm!-90:(p000_l)$);

                 \node[anchor=west,scale=1] (G6) at (1.35,-3.5) {$p_{000}$};
                \node[scale=1.5] (G7) at (-0.15,-3.3) {$\pmb{\underbrace{}{}}$};    
                \draw[line width = {0.05cm}] (-0.2,-3.5) .. controls (-0.2,-3.9) .. (G6);
                \node[scale=3] (G8) at (3.95,-3) {$\pmb{\underbrace{}{}}$};
                \draw[line width = {0.05cm}] (3.85,-3.4) .. controls (4,-3.9) .. (G6);
    
                
                \node[anchor=west,scale=1] (p1_lab) at (-1,-1.8) {$p_{1}$}; 
                \node[anchor=west,scale=1] (p2_lab) at (-1,-0.8) {$p_{2}$}; 
                \node[anchor=west,scale=1] (p3_lab) at (-1,-1.3) {$p_{3}$}; 

            \end{scope}    
        \end{tikzpicture}
    \caption{\textit{Each layer is split into $2$ (i.e. $\binom{2}{1}$) smaller layers. This creates $2$ stacks with $3$ layers each.}
    }
	\label{fig_splittingoverlappinglayers_2}
\end{subfigure}
\par\bigskip
\begin{subfigure}[]{1.0\linewidth}
    \centering
	\begin{tikzpicture}
            \begin{scope}
                \coordinate (r1lu) at (0.3,-1.85);
                \coordinate (r1ld) at (0.3,-1.75);
                \coordinate (r1rd) at (3.7,-1.75);
                \coordinate (r1ru) at (3.7,-1.85);
                \fill[fill=gray!70] (r1lu) -- (r1ld) -- (r1rd) -- (r1ru) -- cycle;

                \coordinate (r2lu) at (0.3,-0.85);
                \coordinate (r2ld) at (0.3,-0.75);
                \coordinate (r2rd) at (3.7,-0.75);
                \coordinate (r2ru) at (3.7,-0.85);
                \fill[fill=gray!70] (r2lu) -- (r2ld) -- (r2rd) -- (r2ru) -- cycle;

                \coordinate (r3lu) at (0.3,-1.35);
                \coordinate (r3ld) at (0.3,-1.25);
                \coordinate (r3rd) at (3.7,-1.25);
                \coordinate (r3ru) at (3.7,-1.35);
                \fill[fill=gray!70] (r3lu) -- (r3ld) -- (r3rd) -- (r3ru) -- cycle;

                \coordinate (r3lu) at (0.3,-1.4);
                \coordinate (r3ld) at (0.3,-1.2);
                \coordinate (r3rd) at (1.4,-1.2);
                \coordinate (r3ru) at (1.4,-1.4);
                \fill[fill=white] (r3lu) -- (r3ld) -- (r3rd) -- (r3ru) -- cycle;

                \coordinate (r2lu) at (1.4,-0.9);
                \coordinate (r2ld) at (1.4,-0.7);
                \coordinate (r2rd) at (2.6,-0.7);
                \coordinate (r2ru) at (2.6,-0.9);
                \fill[fill=white] (r2lu) -- (r2ld) -- (r2rd) -- (r2ru) -- cycle;

                \coordinate (r1lu) at (2.5,-1.9);
                \coordinate (r1ld) at (2.5,-1.7);
                \coordinate (r1rd) at (3.7,-1.7);
                \coordinate (r1ru) at (3.7,-1.9);
                \fill[fill=white] (r1lu) -- (r1ld) -- (r1rd) -- (r1ru) -- cycle;

                \draw[line width=0.07cm,white] (1.4,-0.7) -- (1.4,-2.5);
                
                \draw[line width=0.07cm,white] (2.5,-0.7) -- (2.5,-2.5);





                    
                \draw[-{Latex[length=3mm]},line width=0.06cm] (-1,-2.25) -- (5.85,-2.25); 
                \node[anchor=north,scale=1] (x_axis_label) at (5,-2.65) {$\pmb{1}$}; 

                \node[anchor=north east,scale=1] (x_axis_label) at (-0.55,-2.62) {$\pmb{0}$}; 

                \draw[-{Latex[length=3mm]},line width=0.06cm] (-1,-2.25) -- (5.85,-2.25); 

                \coordinate (p101_l) at (1.45,-3);
                \coordinate (p101_r) at (2.55,-3);
                \draw[|-|,line width=0.06cm] ($(p101_l)!5mm!90:(p101_r)$)--($(p101_r)!5mm!-90:(p101_l)$);
                \node[anchor=west,scale=1] (p101_lab) at (1.6,-2.9) {$p_{101}$}; 

                \coordinate (p110_l) at (0.3,-3);
                \coordinate (p110_r) at (1.4,-3);
                \draw[|-|,line width=0.06cm] ($(p110_l)!5mm!90:(p110_r)$)--($(p110_r)!5mm!-90:(p110_l)$);
                \node[anchor=west,scale=1] (p110_lab) at (0.45,-2.9) {$p_{110}$}; 

                \coordinate (p011_l) at (2.6,-3);
                \coordinate (p011_r) at (3.8,-3);
                \draw[|-|,line width=0.06cm] ($(p011_l)!5mm!90:(p011_r)$)--($(p011_r)!5mm!-90:(p011_l)$);
                \node[anchor=west,scale=1] (p011_lab) at (2.8,-2.9) {$p_{011}$}; 

                \coordinate (p000_l) at (-0.85,-3);
                \coordinate (p000_r) at (0.26,-3);
                \draw[|-|,line width=0.06cm] ($(p000_l)!5mm!90:(p000_r)$)--($(p000_r)!5mm!-90:(p000_l)$);

                \coordinate (p000_l) at (3.85,-3);
                \coordinate (p000_r) at (5,-3);
                \draw[|-|,line width=0.06cm] ($(p000_l)!5mm!90:(p000_r)$)--($(p000_r)!5mm!-90:(p000_l)$);

                 \node[anchor=west,scale=1] (G6) at (1.6,-3.5) {$p_{000}$};
                \node[scale=1.5] (G7) at (-0.15,-3.3) {$\pmb{\underbrace{}{}}$};    
                \draw[line width = {0.05cm}] (-0.2,-3.5) .. controls (-0.2,-3.9) .. (G6);
                \node[scale=1.5] (G8) at (4.5,-3.3) {$\pmb{\underbrace{}{}}$};
                \draw[line width = {0.05cm}] (4.45,-3.5) .. controls (4.5,-3.9) .. (G6);
    
                
                \node[anchor=west,scale=1] (p1_lab) at (-1,-1.8) {$p_{1}$}; 
                \node[anchor=west,scale=1] (p2_lab) at (-1,-0.8) {$p_{2}$}; 
                \node[anchor=west,scale=1] (p3_lab) at (-1,-1.3) {$p_{3}$}; 

            \end{scope}    
        \end{tikzpicture}
    \caption{\textit{From 2 stacks in Fig.~\ref{fig_splittingoverlappinglayers_2}, there are 3 (i.e. $\binom{3}{2}$) stacks created, with each new stack containing 2 layers. One new stack converts part of the event $\cap_{i=1}^{n}A_{i}^c$ (in Figs.~\ref{fig_splittingoverlappinglayers_1}, \ref{fig_splittingoverlappinglayers_2}) into the event $A_{1}^c\cap A_{2}\cap A_{3}$, while the remaining new stacks replace $\cap_{i=1}^{3}A_i$ with events $A_1\cap A_2 \cap A^c_3$ and $A_1\cap A^c_2 \cap A_3$.}
    }
	\label{fig_splittingoverlappinglayers_3}
\end{subfigure}
\caption{}
\label{fig_splittingoverlappinglayers}
\end{figure}

\paragraph{{\textbf{step 2:}}}
\label{subsec_proof_stage2}
Choose an arbitrary ${\mathbb{P}}\in\mathcal D$. If $\mathbb{P}$ induces a stack with more than $k$ of the ``$A$''s, this stack can be split into smaller stacks each consisting of exactly $k$ of the ``$A$''s. The only scenario where such a stack cannot be split is if the distribution has $\cap_{i=1}^{n}A_i^c$ as a null event (i.e. $p_{0...0}=0$). In more detail, if an initial stack has $m>k$ layers, then each layer is split into $\binom{m-1}{k-1}$ horizontal layers from which we construct $\binom{m}{k}$ stacks, or fewer, each with $k$ overlapping layers\footnote{As many $k$-layered stacks as can fit in place of the initial ``$m$ layered'' stack and in place of the probability $p_{0...0}$}. This splitting is illustrated in Fig.~\ref{fig_splittingoverlappinglayers} with $m=3$, $k=2$, and $p_1=p_2=p_3$. Clearly, ``splitting'' does not decrease (and typically increases) the objective function value, since events with no ``$A$''s are replaced with events involving $k$ many ``$A$''s, but not \emph{vice versa}. Hence, restrict $\mathcal D$ to the set ${\mathcal D}^{'}$ of those distributions in $\mathcal D$ that assign non-zero probability only to compound events: \textbf{i)} involving exactly $k$ of the ``$A$''s, or \textbf{ii)} involving only some $A_{n-k+2},\ldots,A_n$, or \textbf{iii)} involving more than $k$ of the ``$A$''s (only when $\cap_{i=1}^{n}A_i^c$ is a null event), or \textbf{iv)} the event $\cap_{i=1}^{n}A_i^c$.

\begin{figure}[!t]
\captionsetup[figure]{format=hang}
    \begin{subfigure}[]{1\linewidth}
        \centering
	\begin{tikzpicture}
            \begin{scope}
                \coordinate (r1lu) at (0.3,-1.85);
                \coordinate (r1ld) at (0.3,-1.75);
                \coordinate (r1rd) at (3.7,-1.75);
                \coordinate (r1ru) at (3.7,-1.85);
                \fill[fill=gray!70] (r1lu) -- (r1ld) -- (r1rd) -- (r1ru) -- cycle;

                \coordinate (r2lu) at (0.3,-0.85);
                \coordinate (r2ld) at (0.3,-0.75);
                \coordinate (r2rd) at (3.7,-0.75);
                \coordinate (r2ru) at (3.7,-0.85);
                \fill[fill=gray!70] (r2lu) -- (r2ld) -- (r2rd) -- (r2ru) -- cycle;

                \coordinate (r3lu) at (0.3,-1.35);
                \coordinate (r3ld) at (0.3,-1.25);
                \coordinate (r3rd) at (3.7,-1.25);
                \coordinate (r3ru) at (3.7,-1.35);
                \fill[fill=white] (r3lu) -- (r3ld) -- (r3rd) -- (r3ru) -- cycle;

                \coordinate (r3lu) at (-0.85,-1.35);
                \coordinate (r3ld) at (-0.85,-1.25);
                \coordinate (r3rd) at (0.2,-1.25);
                \coordinate (r3ru) at (0.2,-1.35);
                \fill[fill=gray!70] (r3lu) -- (r3ld) -- (r3rd) -- (r3ru) -- cycle;

                \coordinate (r3lu) at (-0.85,-1.85);
                \coordinate (r3ld) at (-0.85,-1.75);
                \coordinate (r3rd) at (0.2,-1.75);
                \coordinate (r3ru) at (0.2,-1.85);
                \fill[fill=gray!70] (r3lu) -- (r3ld) -- (r3rd) -- (r3ru) -- cycle;

                \coordinate (r3lu) at (-0.85,-0.85);
                \coordinate (r3ld) at (-0.85,-0.75);
                \coordinate (r3rd) at (0.2,-0.75);
                \coordinate (r3ru) at (0.2,-0.85);
                \fill[fill=white] (r3lu) -- (r3ld) -- (r3rd) -- (r3ru) -- cycle;

                \coordinate (r3lu) at (3.8,-1.35);
                \coordinate (r3ld) at (3.8,-1.25);
                \coordinate (r3rd) at (4.85,-1.25);
                \coordinate (r3ru) at (4.85,-1.35);
                \fill[fill=white] (r3lu) -- (r3ld) -- (r3rd) -- (r3ru) -- cycle;
                
                \coordinate (r3lu) at (2.5,-0.9);
                \coordinate (r3ld) at (2.5,-0.7);
                \coordinate (r3rd) at (3.7,-0.7);
                \coordinate (r3ru) at (3.7,-0.9);
                \fill[fill=white] (r3lu) -- (r3ld) -- (r3rd) -- (r3ru) -- cycle;




                
                \draw[line width=0.07cm,white] (2.5,-0.7) -- (2.5,-2.5);





                    
                \draw[-{Latex[length=3mm]},line width=0.06cm] (-1,-2.25) -- (5.85,-2.25); 
                \node[anchor=north,scale=1] (x_axis_label) at (5,-2.65) {$\pmb{1}$}; 

                \node[anchor=north east,scale=1] (x_axis_label) at (-0.55,-2.62) {$\pmb{0}$}; 

                \draw[-{Latex[length=3mm]},line width=0.06cm] (-1,-2.25) -- (5.85,-2.25); 

                \coordinate (p001_l) at (0.3,-3);
                \coordinate (p001_r) at (2.55,-3);
                \draw[|-|,line width=0.06cm] ($(p001_l)!5mm!90:(p001_r)$)--($(p001_r)!5mm!-90:(p001_l)$);
                \node[anchor=west,scale=1] (p001_lab) at (1.1,-2.9) {$p_{110}$}; 

                \node[anchor=west,scale=1] (p001_lab) at (-0.65,-2.9) {$p_{101}$}; 

                \coordinate (p011_l) at (2.6,-3);
                \coordinate (p011_r) at (3.8,-3);
                \draw[|-|,line width=0.06cm] ($(p011_l)!5mm!90:(p011_r)$)--($(p011_r)!5mm!-90:(p011_l)$);
                \node[anchor=west,scale=1] (p011_lab) at (2.8,-2.9) {$p_{100}$}; 

                \coordinate (p000_l) at (-0.85,-3);
                \coordinate (p000_r) at (0.26,-3);
                \draw[|-|,line width=0.06cm] ($(p000_l)!5mm!90:(p000_r)$)--($(p000_r)!5mm!-90:(p000_l)$);
                \node[anchor=west,scale=1] (p000_lab) at (4,-2.9) {$p_{000}$}; 

                \coordinate (p000_l) at (3.85,-3);
                \coordinate (p000_r) at (5,-3);
                \draw[|-|,line width=0.06cm] ($(p000_l)!5mm!90:(p000_r)$)--($(p000_r)!5mm!-90:(p000_l)$);


                
                \node[anchor=west,scale=1] (p1_lab) at (-1.5,-1.8) {$p_{1}$}; 
                \node[anchor=west,scale=1] (p2_lab) at (-1.5,-0.8) {$p_{2}$}; 
                \node[anchor=west,scale=1] (p3_lab) at (-1.5,-1.3) {$p_{3}$}; 

            \end{scope}    
        \end{tikzpicture}
    \caption{\textit{A Type I distribution that attains ${\mathbb{P}}(\text{at least }2\text{--out--of--}3\text{ events})=\sum_{i=2}^{3}p_i$.}}
	\label{fig_typeIdistribution}
\end{subfigure}
\par\bigskip
\begin{subfigure}[]{1.0\linewidth}
    \centering
	\begin{tikzpicture}
            \begin{scope}
                \coordinate (r1lu) at (0.3,-1.85);
                \coordinate (r1ld) at (0.3,-1.75);
                \coordinate (r1rd) at (3.7,-1.75);
                \coordinate (r1ru) at (3.7,-1.85);
                \fill[fill=gray!70] (r1lu) -- (r1ld) -- (r1rd) -- (r1ru) -- cycle;

                \coordinate (r2lu) at (0.3,-0.85);
                \coordinate (r2ld) at (0.3,-0.75);
                \coordinate (r2rd) at (3.7,-0.75);
                \coordinate (r2ru) at (3.7,-0.85);
                \fill[fill=gray!70] (r2lu) -- (r2ld) -- (r2rd) -- (r2ru) -- cycle;

                \coordinate (r3lu) at (0.3,-1.35);
                \coordinate (r3ld) at (0.3,-1.25);
                \coordinate (r3rd) at (3.7,-1.25);
                \coordinate (r3ru) at (3.7,-1.35);
                \fill[fill=gray!70] (r3lu) -- (r3ld) -- (r3rd) -- (r3ru) -- cycle;

                \coordinate (r3lu) at (-0.85,-1.35);
                \coordinate (r3ld) at (-0.85,-1.25);
                \coordinate (r3rd) at (0.2,-1.25);
                \coordinate (r3ru) at (0.2,-1.35);
                \fill[fill=gray!70] (r3lu) -- (r3ld) -- (r3rd) -- (r3ru) -- cycle;

                \coordinate (r3lu) at (-0.85,-1.85);
                \coordinate (r3ld) at (-0.85,-1.75);
                \coordinate (r3rd) at (0.2,-1.75);
                \coordinate (r3ru) at (0.2,-1.85);
                \fill[fill=gray!70] (r3lu) -- (r3ld) -- (r3rd) -- (r3ru) -- cycle;

                \coordinate (r3lu) at (-0.85,-0.85);
                \coordinate (r3ld) at (-0.85,-0.75);
                \coordinate (r3rd) at (0.2,-0.75);
                \coordinate (r3ru) at (0.2,-0.85);
                \fill[fill=gray!70] (r3lu) -- (r3ld) -- (r3rd) -- (r3ru) -- cycle;

                \coordinate (r3lu) at (3.8,-1.35);
                \coordinate (r3ld) at (3.8,-1.25);
                \coordinate (r3rd) at (4.85,-1.25);
                \coordinate (r3ru) at (4.85,-1.35);
                \fill[fill=gray!70] (r3lu) -- (r3ld) -- (r3rd) -- (r3ru) -- cycle;

                \coordinate (r3lu) at (3.8,-1.85);
                \coordinate (r3ld) at (3.8,-1.75);
                \coordinate (r3rd) at (4.85,-1.75);
                \coordinate (r3ru) at (4.85,-1.85);
                \fill[fill=gray!70] (r3lu) -- (r3ld) -- (r3rd) -- (r3ru) -- cycle;

                \coordinate (r3lu) at (3.8,-0.85);
                \coordinate (r3ld) at (3.8,-0.75);
                \coordinate (r3rd) at (4.85,-0.75);
                \coordinate (r3ru) at (4.85,-0.85);
                \fill[fill=gray!70] (r3lu) -- (r3ld) -- (r3rd) -- (r3ru) -- cycle;
                
                \coordinate (r3lu) at (0.3,-1.4);
                \coordinate (r3ld) at (0.3,-1.2);
                \coordinate (r3rd) at (1.4,-1.2);
                \coordinate (r3ru) at (1.4,-1.4);
                \fill[fill=white] (r3lu) -- (r3ld) -- (r3rd) -- (r3ru) -- cycle;

                \coordinate (r2lu) at (1.4,-0.9);
                \coordinate (r2ld) at (1.4,-0.7);
                \coordinate (r2rd) at (2.5,-0.7);
                \coordinate (r2ru) at (2.5,-0.9);
                \fill[fill=white] (r2lu) -- (r2ld) -- (r2rd) -- (r2ru) -- cycle;

                \coordinate (r1lu) at (2.5,-1.9);
                \coordinate (r1ld) at (2.5,-1.7);
                \coordinate (r1rd) at (3.7,-1.7);
                \coordinate (r1ru) at (3.7,-1.9);
                \fill[fill=white] (r1lu) -- (r1ld) -- (r1rd) -- (r1ru) -- cycle;

                \draw[line width=0.07cm,white] (1.4,-0.7) -- (1.4,-2.5);
                
                \draw[line width=0.07cm,white] (2.5,-0.7) -- (2.5,-2.5);





                    
                \draw[-{Latex[length=3mm]},line width=0.06cm] (-1,-2.25) -- (5.85,-2.25); 
                \node[anchor=north,scale=1] (x_axis_label) at (5,-2.65) {$\pmb{1}$}; 

                \node[anchor=north east,scale=1] (x_axis_label) at (-0.55,-2.62) {$\pmb{0}$}; 

                \draw[-{Latex[length=3mm]},line width=0.06cm] (-1,-2.25) -- (5.85,-2.25); 

                \coordinate (p101_l) at (1.45,-3);
                \coordinate (p101_r) at (2.55,-3);
                \draw[|-|,line width=0.06cm] ($(p101_l)!5mm!90:(p101_r)$)--($(p101_r)!5mm!-90:(p101_l)$);
                \node[anchor=west,scale=1] (p101_lab) at (1.6,-2.9) {$p_{101}$}; 

                \coordinate (p110_l) at (0.3,-3);
                \coordinate (p110_r) at (1.4,-3);
                \draw[|-|,line width=0.06cm] ($(p110_l)!5mm!90:(p110_r)$)--($(p110_r)!5mm!-90:(p110_l)$);
                \node[anchor=west,scale=1] (p110_lab) at (0.45,-2.9) {$p_{110}$}; 

                \coordinate (p011_l) at (2.6,-3);
                \coordinate (p011_r) at (3.8,-3);
                \draw[|-|,line width=0.06cm] ($(p011_l)!5mm!90:(p011_r)$)--($(p011_r)!5mm!-90:(p011_l)$);
                \node[anchor=west,scale=1] (p011_lab) at (2.8,-2.9) {$p_{011}$}; 

                \coordinate (p000_l) at (-0.85,-3);
                \coordinate (p000_r) at (0.26,-3);
                \draw[|-|,line width=0.06cm] ($(p000_l)!5mm!90:(p000_r)$)--($(p000_r)!5mm!-90:(p000_l)$);

                \coordinate (p000_l) at (3.85,-3);
                \coordinate (p000_r) at (5,-3);
                \draw[|-|,line width=0.06cm] ($(p000_l)!5mm!90:(p000_r)$)--($(p000_r)!5mm!-90:(p000_l)$);

                 \node[anchor=west,scale=1] (G6) at (1.6,-3.5) {$p_{111}$};
                \node[scale=1.5] (G7) at (-0.15,-3.3) {$\pmb{\underbrace{}{}}$};    
                \draw[line width = {0.05cm}] (-0.2,-3.5) .. controls (-0.2,-3.9) .. (G6);
                \node[scale=1.4] (G8) at (4.5,-3.3) {$\pmb{\underbrace{}{}}$};
                \draw[line width = {0.05cm}] (4.45,-3.5) .. controls (4.5,-3.9) .. (G6);

                
                \node[anchor=west,scale=1] (p1_lab) at (-1.5,-1.8) {$p_{1}$}; 
                \node[anchor=west,scale=1] (p2_lab) at (-1.5,-0.8) {$p_{2}$}; 
                \node[anchor=west,scale=1] (p3_lab) at (-1.5,-1.3) {$p_{3}$}; 

            \end{scope}    
        \end{tikzpicture}
    \caption{\textit{A Type II distribution that attains ${\mathbb{P}}(\text{at least }2\text{--out--of--}3\text{ events})=1$.}}	\label{fig_typeIIdistribution}
\end{subfigure}
\caption{}
\label{fig_exampledistributiontypes}
\end{figure}

\paragraph{\textbf{step 3:}}
\label{subsec_proof_stage3}
Choose an arbitrary ${\mathbb{P}}\in{\mathcal D}^{'}$. If $\mathbb{P}$ assigns non-zero probability only to stacks with at least $k$ overlaps and, possibly, to the event $\cap_{i=1}^{n}A_i^c$, then $\mathbb{P}$ does not need to be transformed further. However, if $\mathbb{P}$ assigns non-zero probability to an event involving only some $A_{n-k+1},\ldots,A_n$, then $\mathbb{P}$ can be transformed into one of two types of feasible distribution without reducing the objective function value (see Fig.\ \ref{fig_exampledistributiontypes}):
\begin{enumerate}[label=\bfseries\roman*)] 
\item \textbf{(Type I)} all stacks have $k$ layers, and either events $A_{n-r+1},\ldots,A_n$ (for some $1\leqslant r\leqslant k-1$) each have layers in all stacks or no events have layers in all stacks;
\item \textbf{(Type II)} all stacks have at least $k$ layers and $\cap_{i=1}^{n}A_i^c$ is a null event\footnote{A distribution where all stacks have $k$ layers and $\cap_{i=1}^{n}A_i^c$ is a null event is both Type I and Type II.}.
\end{enumerate}

To do this, identify the smallest $j\geqslant (n-k+2)$ such that there is an $A_j$ layer in a stack with less than $k$ overlaps, while there are stacks with at least $k$ overlaps that do not include an $A_j$ layer. Then, some part of the $A_j$ layer can be reassigned, from the ``less-than-$k$'' stack to any ``at least $k$'' stacks that do not involve an $A_j$ layer. If possible, perform similar reassignments of layers for all other ``$A$'' layers in the ``less-than-$k$'' stack. The result of these reassignments is that the only ``$A$'' layers remaining in the ``less-than-$k$'' stack will be for events that now have layers in all stacks with at least $k$ layers. If the distribution has $p_{0...0}>0$, then split any stack with greater than $k$ layers into stacks of $k$ layers and repeat these reassignment steps. After all of these recursive reassignments, if $A_j$ has layers in all stacks with at least $k$ layers, then $A_{j+1}$ must also have layers in all such stacks, since $p_j<p_{j+1}$ implies the ``coverage'' of the layers for $A_j$ must be smaller than the ``coverage'' of layers for $A_{j+1}$. Also, if there remain stacks with less than $k$ layers, the other stacks must have exactly $k$ layers since no more splits can be performed despite $p_{0...0}>0$; because, $A_n$ must be part of all compound events with non-zero probability, so $p_n$ must be an upper bound on the objective function value and $p_{0...0}=1-p_n>0$ (since $p_n<1$). Consequently, restrict \eqref{eqn_optprob} to the set ${\mathcal D}^{''}$ of these Type I and II distributions.    

\paragraph{\textbf{step 4:}}
\label{subsec_proof_stage4}
Let $\Phi=\mathbb{P}(\cup_{I\in{{\boldsymbol I}}:|I|\geqslant k}A_{I})=\sum_{I\in{\boldsymbol I}:|I|\geqslant k}p_I$ be the objective function value for an arbitrary ${\mathbb{P}}\in{\mathcal D}^{''}$. If $\mathbb{P}$ is Type I, all relevant stacks have exactly $k$ layers and $\Phi=\sum_{I\in{\boldsymbol I}:|I|=k}p_I$. Thus, for some $0\leqslant r\leqslant (k-1)$, all $A_i$ (where $i\leqslant n-r$) have layers that contribute only to some, not all, of the stacks with $k$ layers. Relatedly, either $r\geqslant 1$ and there are $r$ ``dominating'' events $A_{n-r+1},\ldots A_n$ that have layers in all $k$-layered stacks, or $r=0$ and no such events. Whichever the case, for all $r$, \[(k-r)\Phi=\!\!\sum_{I\in{\boldsymbol I}:|I|=k}\!\!\!(k-r)p_I=\sum_{I\in{\boldsymbol I}:|I|=k}\!\!\bigg(\sum_{1\leqslant i\leqslant n-r:A_I\subseteq A_i}\!\!p_I\bigg)=\sum_{1\leqslant i\leqslant n-r}\!\!\bigg(\sum_{I\in{\boldsymbol I}:A_I\subseteq A_i}\!\!p_I\bigg)\\=\sum_{i=1}^{n-r}p_i\,,\] since the sum $\sum_{i=1}^{n-r}p_i$ adds each non-zero $p_I$ (when $|I|= k$) a total of ``$k-r$'' times: one time for each layer in the stack for $p_I$, excluding the layers from dominating ``$A$'' events. We conclude, $\Phi=\sum_{I\in{\boldsymbol I}:|I|=k}p_I=\sum_{i=1}^{n-r}p_i/(k-r)$ for Type I distributions. If $\mathbb{P}$ is Type II instead, $\Phi=1$ and $\cap_{i=1}^{n}A_i^c$ is a null event (i.e. $p_{0...0}=0$, since all compound events with non-zero probability involve at least $k$ of the ``$A$''s).

The supremum \emph{must} be \eqref{eqn_soln} when $\sup_{{\mathcal D}^{''}}\Phi<1$. Otherwise, there exists $r\neq r^\ast$ such that $\sup_{{\mathcal D}^{''}}\Phi=\sum_{i=1}^{n-r}p_i/(k-r)$. Such an $r$ cannot exist, because:
\begin{enumerate}[label=\bfseries\roman*)]
\item if $0\leqslant r<r^\ast\!$, \ a Type I distribution must attain the bound. So,
\begin{enumerate}
\item if $r\geqslant 1$, the sets $A_{1}$, $\ldots, A_{n-r}$ each have layers in only some, not all, stacks with $k$ layers; in particular, we find $p_{n-r}<\sum_{i=1}^{n-r}p_i/(k-r)$, hence $p_{n-r}<\sum_{i=1}^{n-r-1}p_i/(k-r-1)$. By the ordering of the $p_i$ probabilities, $p_{n-r-1}<p_{n-r}<\sum_{i=1}^{n-r-1}p_i/(k-r-1)$, so that $p_{n-r-1}<\sum_{i=1}^{n-r-1}p_i/(k-r-1)$. Using this form of argument recursively gives the final inequality $p_{n-r^\ast+1}<\sum_{i=1}^{n-r^\ast+1}p_i/(k-r^\ast+1)$, which contradicts the definition of $r^\ast$; 
\item if $r=0$ there are no dominating events, and the supremum takes the form $\sum_{i=1}^{n}p_i/k$. Therefore, $p_{n}\leqslant \sum_{i=1}^{n}p_i/k$, which recursively implies $p_{n-r^\ast+1}< \sum_{i=1}^{n-r^\ast+1}p_i/(k-r^\ast+1)$ --- another contradiction of $r^\ast$'s definition. 
\end{enumerate}
\item if $0\leqslant r^\ast\!\!<r$, a Type I distribution attains the bound, the event $A_{n-r+1}$ is a dominating event, and $p_{n-r+1}\geqslant\sum_{i=1}^{n-r}p_i/(k-r)$. That is, $p_{n-r+1}\geqslant\sum_{i=1}^{n-r+1}p_i/(k-r+1)$, which contradicts $r^*$ as the largest $r$ satisfying such inequalities.
\end{enumerate}

Finally, the supremum must still be \eqref{eqn_soln} when $\sup_{{\mathcal D}^{''}}\Phi=1$. To see this, prove that $\sup_{{\mathcal D}^{''}}\Phi=1$ \emph{iff} $r^\ast$ satisfies $\min\{\sum_{i=1}^{n-r^\ast}p_i/(k-r^\ast),1\}=1$, as follows. 

Sufficiency: Suppose $\sup_{{\mathcal D}^{''}}\Phi=1$. This is attained by a Type II distribution: all stacks consist of $k$, or more, layers and $p_{0...0}=0$. Splitting those stacks with more than $k$ layers produces all stacks having exactly $k$ layers --- this gives $\sum_{i=1}^{n}p_i/k$ as an upper bound of $\Phi$, so this upper bound is greater than $\Phi$'s least upper bound, 1. Therefore, we must have $p_n<1<\sum_{i=1}^{n}p_i/k$. Rearranging, then using the ordering of the $p_i$ probabilities, gives $p_{n-1}<\sum_{i=1}^{n-1}p_i/(k-1)$. By applying this form of argument recursively, we deduce none of $p_n,\ldots,p_{n-k+2}$ satisfy the inequality for the ``$\max$'' in the definition of $r^\ast$. Hence, $r^\ast=0$ and $\sum_{i=1}^{n-r^\ast}p_i/(k-r^\ast)=\sum_{i=1}^{n}p_i/k$. Consequently, \eqref{eqn_soln} correctly states the supremum as  $\sup_{{\mathcal D}^{''}}\Phi=\min\{\sum_{i=1}^{n}p_i/k,1\}=1$. 

Necessity: Suppose $\sum_{i=1}^{n-r^\ast}p_i/(k-r^\ast)>1$ instead. The definition of $r^\ast$ implies $r^\ast=0$; otherwise, $p_{n-r^\ast+1}\geqslant \sum_{i=1}^{n-r^\ast}p_i/(k-r^\ast)>1$. It follows that $\sum_{i=1}^{n}p_i/k>1$. This inequality means $\sum_{i=1}^{n}p_i/k$ cannot be attained by any feasible distribution. Moreover, since both $r^\ast=0$ and this inequality hold, the closest $\Phi$ value to $\sum_{i=1}^{n}p_i/k$ must be $1$, from a Type II distribution consisting of only stacks with at least $k$ layers and $p_{0...0}=0$. Hence, $\sup_{{\mathcal D}^{''}}\Phi=1$.\qedhere    



\end{proof}

\begin{corollary}
\label{cor_inf_1}
Let $k\in\{0,\ldots,n-1\}$. Then,
\begin{align}
\label{eqn_optprob_inf_1}
&\inf\limits_{\mathbb{P}} \mathbb{P}(\text{at most }k\text{--out--of--}n\text{ events}) \nonumber\\
\text{s.t.}\,&\,\, \mathbb{P}(A_1)=p_1,\ldots,\mathbb{P}(A_n)=p_n
\end{align}
has the solution 
\begin{align}
\label{eqn_soln_inf_1}
\max\!\left\{1-\sum_{i=1}^{n-r^\ast}\frac{p_i}{k+1-r^\ast},\,0\right\}
\end{align}
where 
$$r^\ast:=\left\{\begin{array}{ll}\max\!\left\{r\in\{1,\ldots,k\}\,\,\middle|\,\, p_{n-r+1} \geqslant \sum\limits_{i=1}^{n-r+1}\left(\frac{p_i}{k-r+2}\right)\right\}\!,&\text{if max exists} \\ 0,&\text{otherwise}\end{array}\right.$$
\end{corollary}
\begin{proof}
$$\inf\limits_{\mathbb{P}} \mathbb{P}(\text{at most }k\text{--out--of--}n\text{ events})=1-\sup\limits_{\mathbb{P}} \mathbb{P}(\text{at least }(k+1)\text{--out--of--}n\text{ events})\qedhere$$
\end{proof}

\begin{corollary}
\label{cor_inf_2}
Let $k\in\{1,\ldots,n\}$. Then,
\begin{align}
\label{eqn_optprob_inf_2}
&\inf\limits_{\mathbb{P}} \mathbb{P}(\text{at least }k\text{--out--of--}n\text{ events}) \nonumber\\
\text{s.t.}\,&\,\, \mathbb{P}(A_1)=p_1,\ldots,\mathbb{P}(A_n)=p_n
\end{align}
has the solution 
\begin{align}
\label{eqn_soln_inf_2}
\max\!\left\{1-\!\!\!\sum_{i=r^\ast+1}^{n}\frac{1-p_i}{n-k+1-r^\ast}\,\,,\,0\right\}
\end{align}
where 
$$r^\ast:=\left\{\!\!\begin{array}{ll}\max\!\left\{r\in\{1,\ldots,n-k\}\,\,\middle|\,\, (1-p_{r}) \geqslant \sum\limits_{i=r}^{n}\left(\frac{1-p_i}{n-k-r+2}\right)\right\}\!,&\text{if max exists} \\ 0,&\text{otherwise}\end{array}\right.$$
\end{corollary}

\begin{proof}
$$\inf\limits_{\mathbb{P}} \mathbb{P}(\text{at most }(n-k)\text{--out--of--}n\text{ do not occur})=\inf\limits_{\mathbb{P}} \mathbb{P}(\text{at least }k\text{--out--of--}n\text{ events})$$
Thus, in \eqref{eqn_optprob_inf_1} and \eqref{eqn_soln_inf_1}, substitute $(1-p_{n-i+1})$ for $p_i$, and $(n-k)$ for $k$.\qedhere
\end{proof}
\begin{corollary}
\label{cor_sup_2}
Let $k\in\{0,\ldots,n-1\}$. Then,
\begin{align}
\label{eqn_optprob_sup_2}
&\sup\limits_{\mathbb{P}} \mathbb{P}(\text{at most }k\text{--out--of--}n\text{ events}) \nonumber\\
\text{s.t.}\,&\,\, \mathbb{P}(A_1)=p_1,\ldots,\mathbb{P}(A_n)=p_n
\end{align}
has the solution 
\begin{align}
\label{eqn_soln_sup_2}
\min\!\left\{\sum_{i=r^\ast+1}^{n}\frac{1-p_i}{n-k-r^\ast}\,\,,\,1\right\}
\end{align}
where 
$$r^\ast:=\left\{\begin{array}{ll}\max\!\left\{r\in\{1,\ldots,n-k-1\}\,\,\middle|\,\, 1-p_{r} \geqslant \sum\limits_{i=r}^{n}\!\left(\frac{1-p_i}{n-k-r+1}\right)\right\}\!,&\text{if max exists} \\ 0,&\text{otherwise}\end{array}\right.$$
\end{corollary}
\begin{proof}
\begin{align*}
\sup\limits_{\mathbb{P}} \mathbb{P}(\text{at most }k\text{--out--of--}n\text{ events})=&\,\sup\limits_{\mathbb{P}} \mathbb{P}(\text{at least }(n-k)\text{--out--of--}n\text{ do not occur})\\
=&\,1-\inf\limits_{\mathbb{P}} \mathbb{P}(\text{at least }(k+1)\text{--out--of--}n\text{ events})\,\qedhere
\end{align*}
\end{proof}

\subsection{General Construction of Sharp Bounds: A Dynamic Programming Problem.}
\label{subsec_DPformulation}
The constructive approach of section~\ref{subsec_constructedkoutofnbounds} can be generalized as policy improvements for the following \emph{deterministic}, \emph{finite-horizon}, \emph{undiscounted} dynamic programming (DP) problem. 
As before, stacks are placed over the unit interval $[0,1]$. A vertical cut of the unit interval starts at $t=1$ and moves leftward to $t=0$. When the cut is at any $t\in(0,1)$ it separates the interval into two segments: {\textbf{i)}} a \emph{left} segment $L(t)$ of length $t$ with stacks that are yet to be transformed (if needed), and {\textbf{ii)}} a \emph{right} segment $R(t)$ of length $1-t$ that has stacks constructed out of stacks from earlier $L(t)$ segments. So, as $t$ moves, stacks over $L(t)$ are transformed into stacks over $R(t)$. At cut $t$, the portion of stack $I$ over $R(t)$ is denoted $z_I$~---~the $I$-th component of a vector ${\boldsymbol z}\in \mathbb{R}^{|\boldsymbol I|}_{\geqslant 0}$, where $|\boldsymbol I|=2^n$. 

\paragraph{State variables.}
At any cut $t$, there is a vector ${\boldsymbol l}$ with $i$-th component $l_i$ that is the left amount of $p_i$ over the segment $L(t)$. Thus, $l_i$ is the \emph{Lebesgue measure} within $L(t)$ on which $A_i=1$, $l_i\leqslant p_i$, and the amount $p_i-l_i$ lies over $R(t)$. The \emph{state} at cut $t$ is a triplet, $(t,{\boldsymbol l}, {\boldsymbol z})$, from the \emph{state-space}
\[
\mathcal S:=\bigl\{(t,{\boldsymbol l}, {\boldsymbol z}):\, t\in[0,1],\ \ \sum_{I\in{\boldsymbol I}}z_I=1-t,\ \ l_i\in[0,\min\{t,p_i\}]\text{ for }i=1,\dots,n,\ {\boldsymbol l}+M{\boldsymbol z}={\boldsymbol p})\bigr\},
\]
where $t$ is the \emph{length} of $L(t)$ and $M\in\{0,1\}^{n\times|\boldsymbol I|}$ is a matrix whose columns are the binary words $I\in\boldsymbol I$. The initial state is $(t,{\boldsymbol l},{\boldsymbol z})=(1,{\boldsymbol p},{\boldsymbol 0})$, where the $p_i$ are the components of ${\boldsymbol p}$.

\paragraph{Policies.}
A \emph{policy} $\pi$ is a collection of deterministic transformation rules for stacks~---~one rule per state~---~that map $\mathcal S$ to $\mathcal S$ while satisfying the following feasibility requirements. 

At cut $t$, the rule is defined for any slice from $L(t)$ of length $\Delta=l_i$ for some $l_i>0$, otherwise of length $\Delta=t$ if there are no $l_i>0$. The slice defines a vector, ${\boldsymbol x}(t)\in \mathbb{R}^{|\boldsymbol I|}_{\geqslant 0}$, whose $I$-th component, $x_{I}(t)$, is the amount of stack $I$ lying over the $\Delta$ slice; so, $\sum_{I\in\boldsymbol I} x_{I}(t)=\Delta$. The component $u_i$ of the vector ${\boldsymbol u}(t):=M{\boldsymbol x}(t)$ is the amount of $p_i$ over the $\Delta$ slice. The rule transforms the ${\boldsymbol x}(t)$ stacks over the $\Delta$ slice in $L(t)$ to ${\boldsymbol y}(t)$ stacks over $R(t-\Delta)$, where: {\textbf {i)}} ${\boldsymbol y}(t)\in \mathbb{R}^{|\boldsymbol I|}_{\geqslant 0}$, {\textbf{ii)}} $\sum_{I\in\boldsymbol I} y_{I}(t)=\Delta$, and {\textbf{iii)}} $M{\boldsymbol x}(t)={\boldsymbol u}(t)=M{\boldsymbol y}(t)$. Under the rule, $(t,{\boldsymbol l}, {\boldsymbol z})\mapsto (t-\Delta,{\boldsymbol l}-{\boldsymbol u}(t),{\boldsymbol y}(t)+{\boldsymbol z})$. From the definitions of $M$, ${\boldsymbol x}(t)$, ${\boldsymbol y}(t)$, and ${\boldsymbol z}$, it follows that $\max\{{\boldsymbol 0}, {\boldsymbol l}-(t-\Delta){\boldsymbol 1}\}\leqslant {\boldsymbol u}(t)\leqslant {\boldsymbol l}$ and ${\boldsymbol y}(t)+{\boldsymbol z}\leqslant (1-t+\Delta){\boldsymbol 1}$ componentwise (where $\boldsymbol{1}$, $\boldsymbol{0}$, are vectors of ones and zeroes respectively, of appropriate dimension). In summary, the triplet $(\Delta, {\boldsymbol x}(t), {\boldsymbol y}(t))$ is the transformation rule or \emph{action} taken in state $(t,{\boldsymbol l},{\boldsymbol z})$ by following policy $\pi$. 



\paragraph{Rewards.}
 A compound event $E$ is specified by a non-empty subset $E\subseteq \boldsymbol{I}$. Suppose we seek $\Phi^*=\sup\mathbb P(E)$ subject to the marginal $\boldsymbol p$ constraints. The \emph{reward}, obtained by following policy $\pi$ at vertical cut $t$, is $e({\boldsymbol y}):=\sum_{I\in E} y_I$; this is the marginal increase in $\mathbb P(E)$ on $R(t-\Delta)$ due to transforming stacks over $L(t)$ by following policy $\pi$. The total reward at vertical cut $t$ is the sum of all stack probabilities over $R(t)$ for those stacks indexed in $E$;  stacks over $L(t)$ do not contribute to the total reward.

\paragraph{Bellman structure.}
Define the Bellman operator $\mathcal{T}$ on bounded $W:\mathcal{S}\to[0,1]$,
\begin{equation}
\label{eq:Bellman-operator}
(\mathcal{T}W)(t,{\boldsymbol l},{\boldsymbol z}):=\sup_{\substack{(\Delta,{\boldsymbol x}(t), {\boldsymbol y}(t))}}
\Bigl\{\, e({\boldsymbol y(t)})\ +\ W\bigl(t-\Delta,\ {\boldsymbol l}-{\boldsymbol u}(t),\ {\boldsymbol y}(t)+{\boldsymbol z}\bigr)\,\Bigr\},
\qquad W(0,{\boldsymbol 0},{\boldsymbol z}):=0 .
\end{equation}




This Bellman operator has a fixed-point that is the solution (on $\mathcal S$) of the statewise Linear Programming (LP) problem \eqref{eqn:statewiseLP}.  
\begin{lemma}[Statewise LP problem]
\label{lem:statewise-Phi}
For each $(t,{\boldsymbol l},{\boldsymbol z})\in\mathcal{S}$, define
\begin{equation}
\Phi(t,{\boldsymbol l},{\boldsymbol z}):=\max\Bigl\{\ \sum_{I\in E} y_I\ :\ {\boldsymbol y}\in\mathbb{R}^{|\boldsymbol I|}_{\geqslant 0},\ \sum_{I\in{\boldsymbol I}} y_I=t,\ \ M{\boldsymbol y}={\boldsymbol l}\ \Bigr\}.
\label{eqn:statewiseLP}
\end{equation}
Then $\Phi(1,{\boldsymbol p},{\boldsymbol 0})=\Phi^*$, and $\Phi(t,{\boldsymbol l},{\boldsymbol z})$ is well-defined and attained for all $(t,{\boldsymbol l},{\boldsymbol z})\in\mathcal{S}$.
\end{lemma}

\begin{proof}
Feasibility: This LP problem does not depend on the particular ${\boldsymbol z}$~---~any ${\boldsymbol z}$ consistent with ${\boldsymbol l}$ defines the same LP. Since $(t,{\boldsymbol l}, {\boldsymbol z})\in\mathcal{S}$, the ${\boldsymbol l}$ and ${\boldsymbol z}$ vectors are consistent with the ${\boldsymbol p}$ constraints. Split the ``$A_i$'' layer into two layers with lengths $l_i$ and $p_i-l_{i}$. Place all of the $l_i$-length layers over $L(t)$, with the $p_i-l_i$ layers over $R(t)$ contributing to ${\boldsymbol z}$ stacks. Set $y_I$ to be the length of stack $I$ over $L(t)$. Then, $\boldsymbol y$ satisfies the constraints of \eqref{eqn:statewiseLP}.

 Attainment: Compactness of the set of feasible ${\boldsymbol y}$ in \eqref{eqn:statewiseLP} and the linearity (continuity, in particular) of the objective $\sum_{I\in E}y_I$ imply attainment. Taking $(t,{\boldsymbol l}, {\boldsymbol z})=(1,{\boldsymbol p},{\boldsymbol 0})$ the solution to the LP is a stack arrangement over $[0,1]$ that maximizes the total horizontal length of stacks representing the occurrence of event $E$. This is precisely $\Phi^*$.
\end{proof}

\begin{lemma}[Existence of Bellman-operator fixed point]
\label{lem:decomposition}
For all $(t,{\boldsymbol l},{\boldsymbol z})\in\mathcal{S}$,
\[
\Phi(t,{\boldsymbol l},{\boldsymbol z})=\sup_{\substack{(\Delta,{\boldsymbol x}(t)},{\boldsymbol y}(t))} \Bigl\{\, e({\boldsymbol y}(t))\ +\ \Phi\bigl(t-\Delta,\ {\boldsymbol l}-M{\boldsymbol y}(t),\ {\boldsymbol y}(t)+{\boldsymbol z}\bigr)\,\Bigr\}.
\]
Equivalently, $\Phi$ is a fixed point of the Bellman operator: $\Phi=\mathcal{T}\Phi$ on $\mathcal{S}$.
\end{lemma}

\begin{proof}
For $t=0$ the equality is immediate. For $t>0$, there are two parts to the proof.

\emph{($\geqslant$)} Let $(\Delta,{\boldsymbol y},{\boldsymbol y})$ be a stack-transformation choice\footnote{Given any feasible ${\boldsymbol y}$ in the lemma's constraint set, we choose ${\boldsymbol x}={\boldsymbol y}$; this is admissible because the only constraints on ${\boldsymbol x}$ are $\sum_{I\in{\boldsymbol I}}{\boldsymbol x}_I=\Delta$, ${\boldsymbol x}\geqslant{\boldsymbol 0}$, and $M{\boldsymbol x}=M{\boldsymbol y}$.} at $(t,{\boldsymbol l},{\boldsymbol z})$ and, after enacting this choice, let ${\boldsymbol y}'$ be any feasible remainder with $\sum_{I\in{\boldsymbol I}} y_I'=t-\Delta$ and $M{\boldsymbol y}'={\boldsymbol l}-M{\boldsymbol y}$. Then ${\boldsymbol y}{+}{\boldsymbol y}'$ is feasible for $\Phi(t,{\boldsymbol l},{\boldsymbol z})$ and, by definition~\eqref{eqn:statewiseLP},
\[
\Phi(t,{\boldsymbol l},{\boldsymbol z})\geqslant\sum_{I\in E}(y_I+y'_I)=e({\boldsymbol y})+\sum_{I\in E}y'_I= e({\boldsymbol y})+\Phi(t-\Delta,{\boldsymbol l}-M{\boldsymbol y},{\boldsymbol y}+{\boldsymbol z}),
\]
if we choose ${\boldsymbol y}^\prime$ such that $\sum_{I\in E}y'_I=\Phi(t-\Delta,{\boldsymbol l}-M{\boldsymbol y},{\boldsymbol y}+{\boldsymbol z})$. Taking the supremum over $(\Delta,{\boldsymbol y},{\boldsymbol y})$ yields $\Phi(t,{\boldsymbol l},{\boldsymbol z})\geqslant (\mathcal{T}\Phi)(t,{\boldsymbol l},{\boldsymbol z})$.

\emph{($\leqslant$)} Let ${\boldsymbol y}^*$ be any optimizer for $\Phi(t,{\boldsymbol l},{\boldsymbol z})$. Fix admissible $\Delta$ (by construction, $\Delta\in(0,t]$), set ${\boldsymbol y}:=\frac{\Delta}{t}{\boldsymbol y}^*$ and ${\boldsymbol y}':={\boldsymbol y}^*-{\boldsymbol y}$. Then $\sum_{I\in{\boldsymbol I}} y_I=\Delta$, $M{\boldsymbol y}=\frac{\Delta}{t}\,{\boldsymbol l}$, $\sum_{I\in{\boldsymbol I}} y_I'=t-\Delta$, $M{\boldsymbol y}'={\boldsymbol l}-M{\boldsymbol y}$. The transformation choice $(\Delta,{\boldsymbol y},{\boldsymbol y})$ is admissible at cut $t$, therefore
\[
\Phi(t,{\boldsymbol l},{\boldsymbol z})=\sum_{I\in E}{\boldsymbol y}^*_I=e({\boldsymbol y})+\sum_{I\in E}{\boldsymbol y}'_I\leqslant e({\boldsymbol y})+\Phi(t-\Delta,{\boldsymbol l}-M{\boldsymbol y},{\boldsymbol y}+{\boldsymbol z}).
\]
Taking the supremum over admissible $(\Delta,{\boldsymbol y},{\boldsymbol y})$ gives $\Phi(t,{\boldsymbol l},{\boldsymbol z})\leqslant (\mathcal{T}\Phi)(t,{\boldsymbol l},{\boldsymbol z})$.
\end{proof}



\begin{theorem}[Bellman fixed-point exists and equals the LP value]
\label{thm:bellmanconvergence}
There exists a bounded function $W^*:\mathcal{S}\to[0,1]$ such that
\[
W^*=\mathcal{T}W^*,\,\,\, W^*(0,{\boldsymbol 0},{\boldsymbol z})=0.
\]
Moreover, $W^*=\Phi$ on $\mathcal{S}$, and $W^*(1,{\boldsymbol p},{\boldsymbol 0})=\Phi^*$.
\end{theorem}

\begin{proof}
Let $W^*$ denote the DP value function. Consider following a policy starting from $(t,{\boldsymbol l},{\boldsymbol z})$. Let vector ${\boldsymbol m}\in\mathbb{R}^{|\boldsymbol I|}_{\geqslant 0}$ contain only those accumulated ${\boldsymbol y}$ stack lengths resulting from applying the policy to stacks in $(0,t]$. Then, $\sum_{I\in{\boldsymbol I}} m_I=t$ and $M{\boldsymbol m}={\boldsymbol l}$. The total $E$-mass contributed by following the policy, $\sum_{I\in E} m_I$, is bounded: $\sum_{I\in E} m_I \leqslant \Phi(t,{\boldsymbol l},{\boldsymbol z})$.
Thus, upon taking the supremum over all policies, $W^*(t,{\boldsymbol l},{\boldsymbol z})\leqslant \Phi(t,{\boldsymbol l},{\boldsymbol z})$.

Conversely, for any optimizer ${\boldsymbol y}^*$ of $\Phi(t,{\boldsymbol l},{\boldsymbol z})$, there is a sequence of admissible moves $(\Delta_1,{\boldsymbol y}^*_1,{\boldsymbol y}^*_1),\ldots,\allowbreak(\Delta_k,{\boldsymbol y}^*_k,{\boldsymbol y}^*_k)$ for some $k\leqslant n+1$, such that $t=\sum_{i=1}^k\Delta_i$,  ${\boldsymbol y}^*=\sum_{i=1}^k{\boldsymbol y}^*_i$ and $\sum_{i=1}^{k}e({\boldsymbol y}^*_i)=\Phi(t,{\boldsymbol l},{\boldsymbol z})$. This can be constructed by choosing ${\boldsymbol y}^*_1$ to be all stacks in ${\boldsymbol y}^*$ that satisfy $A_{j_1}=1$ for some $j_1$, choosing $\Delta_1$ to be the sum of these stacks' horizontal lengths, then repeating analogous constructions for the remaining $(\Delta_{i},{\boldsymbol y}^*_i,{\boldsymbol y}^*_i)$ triplets based on any remaining $A_{j_i}$ layer contributions. If no $A_{j_i}=1$ stacks remain and $t>0$, choose the $I=0...0$ stack to construct a triplet. This construction produces admissible transformation triplets:\textbf{i)} $\Delta_1=\sum_{I\in{\boldsymbol I}}(y^*_1)_I=l_{j_1}\leqslant t$; \textbf{ii)} ${\boldsymbol u}_1:=M{\boldsymbol y}^*_1$ satisfies ${\boldsymbol 0}\leqslant {\boldsymbol u}_1\leqslant {\boldsymbol l}$ and ${\boldsymbol l}-{\boldsymbol u}_1=M{\boldsymbol y}^\prime_1\leqslant (t-\Delta_1){\boldsymbol 1}$, so the inequality $\max\{{\boldsymbol 0},{\boldsymbol l-(t-\Delta_1){\boldsymbol 1}}\leqslant {\boldsymbol u}_1\leqslant {\boldsymbol l}\}$ holds, so the triplet $(\Delta_1,{\boldsymbol y}^*_1,{\boldsymbol y}^*_1)$ is admissible; \textbf{iii)} at the next state $(t-\Delta_1,{\boldsymbol l}-{\boldsymbol u}_1,{\boldsymbol z}+{\boldsymbol y}^*_1)$ the remainder ${\boldsymbol y}^\prime_1$ is feasible in the sense of Lemma~\ref{lem:statewise-Phi}, so repeat the admissibility argument for the successive transformation triplets. Hence, $W^*(t,{\boldsymbol l},{\boldsymbol z})\geqslant \Phi(t,{\boldsymbol l},{\boldsymbol z})$. 

Since $(t,{\boldsymbol l},{\boldsymbol z})\in\mathcal S$ was arbitrary, we conclude $W^*=\Phi$ on $\mathcal S$. In particular, $W^*(1,{\boldsymbol p},{\boldsymbol 0})=\Phi(1,{\boldsymbol p},{\boldsymbol 0})=\Phi^*$ by Lemma~\ref{lem:statewise-Phi}.\qedhere

\end{proof}

\begin{theorem}[Value-iteration converges to LP value]
\label{thm:DPvalueiterationconvergence}
The sequence of functions $W_k:\mathcal S\rightarrow[0,1]$, $W_{k+1}:=\mathcal T W_k$, $W_0=0$ on $\mathcal S$, converges monotonically to $\Phi$.
\end{theorem}
\begin{proof}
For bounded $W,\,V$, if $W\leqslant V$ on $\mathcal S$, then $\mathcal T W = \sup\{e({\boldsymbol y})+W\}\leqslant \sup\{e({\boldsymbol y})+V\}=\mathcal T V$ schematically. Note that $W_0=0\leqslant \mathcal T W_0=W_1$, which implies $W_1=\mathcal T W_0\leqslant \mathcal T W_1 = W_2$; in general, by induction, $W_k\leqslant W_{k+1}$ for all $k$. Similarly, $W_0=0\leqslant\Phi$ implies $W_1=\mathcal T W_0\leqslant\mathcal T\Phi=\Phi$; induction proves $W_k\leqslant\Phi$ for all $k$. Thus, the bounded, monotonically increasing $W_k$, have a finite limit $\sup_k W_k$. Any policy that accrues total reward $\sum_{i=1}^{k}e({\boldsymbol y}_i)$ in $k$ steps (from admissible ``${\boldsymbol y}_i$''s) is dominated by $W_k$: $W_k=\sup\{e({\boldsymbol y})+W_{k-1}\}\geqslant e({\boldsymbol y}_k)+W_{k-1}\geqslant e({\boldsymbol y}_k)+\sup\{e({\boldsymbol y})+W_{k-2}\}\geqslant \sum_{i=k-1}^k e({\boldsymbol y}_i)+W_{k-3}\geqslant\ldots\geqslant\sum_{i=1}^k e({\boldsymbol y}_i)+W_0=\sum_{i=1}^k e({\boldsymbol y}_i)$. Thus, $\Phi=W^*\leqslant \sup_kW_k$. Conversely, $W_k\leqslant W^*=\Phi$ for all $k$. Thus, $\sup_k W_k\leqslant \Phi$. We conclude, $\sup_k W_k=\Phi$ on $\mathcal S$.
\end{proof}

\begin{corollary}[Finite-time value-iteration]
\label{cor:nplus1stepstophi}
 $W_{k}=\Phi$ on $\mathcal S$,  for all $k\geqslant n+1$.
\end{corollary}
\begin{proof}
Follows from Theorems~\ref{thm:bellmanconvergence} (there exists a policy that achieves the optimum in at most $n+1$ steps) and \ref{thm:DPvalueiterationconvergence} (so, the value of $n+1$ steps, $W_{n+1}$, is the DP/LP value, which is a fixed point of the Bellman operator; all higher step values are also the DP/LP value).
\end{proof}


\section{Discussion. }
\label{sec_discussion}
\paragraph{Remarks.}
\label{subsec_remarks}
The bounds assume strict inequalities between the $p_i$ probabilities of the ``$A_i$'' events; bounds for weaker inequalities are simple limiting cases of these bounds.
 
The proof of Theorem\ \ref{thrm_sup}'s bounds can be viewed in terms of the DP problem underpinning Theorem~\ref{thm:bellmanconvergence}. 
Formally, when $E = E_k := \{I \in \boldsymbol{I} : |I| \geqslant k\}$, the constructive steps of Theorem~\ref{thrm_sup} are three nested policy improvements. For any admissible policy $\pi$, let $\mathbb P_\pi$ denote its terminal stack-configuration distribution at time $t=0$, and write $\Phi_k(\mathbb P) := \mathbb P\big(\bigcup_{I\in E_k} A_I\big)$ for the corresponding objective. Step~1 defines a post-processing policy $\Pi_1$ which, applied to $\mathbb P_\pi$, rearranges layers to produce a distribution $\mathbb P_\pi^{(1)} \in \mathcal D$ (as defined in Theorem~\ref{thrm_sup}) that preserves the marginals and satisfies $\Phi_k(\mathbb P_\pi^{(1)}) \geqslant \Phi_k(\mathbb P_\pi)$. Step~2 defines a second post-processing policy $\Pi_2$ which, for any input $\mathbb P \in \mathcal D$, splits stacks with more than $k$ layers to obtain $\mathbb P^{(2)} \in \mathcal D'$, again with unchanged marginals and $\Phi_k(\mathbb P^{(2)}) \geqslant \Phi_k(\mathbb P)$. Step~3 defines a third post-processing policy $\Pi_3$ which, for any $\mathbb P \in \mathcal D'$, reassigns layers from stacks with fewer than $k$ overlaps to stacks with at least $k$ overlaps whenever possible, and thereby reaches a Type~I or Type~II distribution $\mathbb P^{(3)} \in \mathcal D''$ with $\Phi_k(\mathbb P^{(3)}) \geqslant \Phi_k(\mathbb P)$. Thus, starting from the terminal distribution of any original policy $\pi$, the composite policy $\Pi_3 \circ \Pi_2 \circ \Pi_1 \circ \pi$ yields a terminal distribution in $\mathcal D''$ whose value is at least $W^\pi(1,\boldsymbol p,\boldsymbol 0)$, so that $\sup_\pi W^\pi(1,\boldsymbol p,\boldsymbol 0) = \sup_{\mathbb P \in \mathcal D''} \Phi_k(\mathbb P)$. By Theorem~\ref{thm:bellmanconvergence}, this supremum equals the optimal value $W^*(1,\boldsymbol p,\boldsymbol 0)$, and by Step~4 of Theorem~\ref{thrm_sup} the value of $\Phi_k$ on $\mathcal D''$ is the closed form in Theorem~\ref{thrm_sup}. For general $E$, Theorem~\ref{thm:DPvalueiterationconvergence} suggests (for each $k\leqslant n+1$) a policy $\pi_k$, with value $W_k$, that dominates all policies over $k$ steps.

Theorem~\ref{thm:bellmanconvergence} relies on feasible distributions that satisfy $\{ {\boldsymbol y} \geqslant 0 : \sum_{I\in{\boldsymbol I}} y_I = 1,\; M{\boldsymbol y} = {\boldsymbol p}\}$. If, in addition to the single--event marginals
$\boldsymbol p$, we are given (exact or interval) constraints on the probabilities
of finitely many further Boolean formulae in the ``$A_i$'' events~---~of the form
$\mathbb P(F_j) = q_j$ or $a_j \leqslant \mathbb P(F_j) \leqslant b_j$ for Boolean formulae
$F_1,\dots,F_m$~---~we may introduce extra rows in $M$ for these formulae and extend the vector $\boldsymbol p$ accordingly.  The same Bellman operator
(with the state space enlarged by the new coordinates) and the same fixed--point
argument then apply verbatim.  Thus, the DP framework provides
a constructive solution for sharp bounds under any finite system of such linear
constraints on probabilities of Boolean events. 
In trivial cases (e.g. $E=\emptyset$ or ${\boldsymbol I}$) the sharp bound is $0$ or $1$, and the extremal distributions need not be unique.

Theorem~\ref{thm:bellmanconvergence} gives an equivalence between a DP problem and an LP problem. Classical DP/LP equivalence results go back to Manne and many later authors \cite{manne1960,buyuktahtakin2011}, who show that Markov decision processes can be reformulated as linear programs over value-function or occupation-measure variables subject to Bellman-type inequalities, with the LP optimum equal to the Bellman fixed point for the DP. Theorem~\ref{thm:bellmanconvergence} is analogous in spirit, but distinct: it identifies the Bellman fixed point, $W^*$, with the optimum of the Boole--Fr\'{e}chet LP~\eqref{eqn:statewiseLP}, whose decision variables are joint stack-distributions $\boldsymbol y$ constrained only by probability and marginal conditions. This LP structure and feasible set differ substantially from standard Bellman-inequalities or occupation-measure formulations.

Conceptually, this DP framework is one way of looking at extremal rearrangements under fixed marginals, and many sharp bounds it yields could, in principle, also be approached via copula constructions such as ``shuffles-of-$M$'' \cite{DuranteSarkociSempi2009,MikusinskiSherwoodTaylor1992}.

\paragraph{Bounds on Probabilities of exactly ``$k$--out--of--$n$'' Events.}
\label{subsec_extensions}

The upper bounds for ${\mathbb{P}}(\text{at least }\allowbreak k\text{--out--of--}\allowbreak n\text{ events})$ are also the upper bounds for ${\mathbb{P}}(k\text{--out--of--}n\text{ events})$ (i.e. $\mathbb{P}(\cup_{I\in{\boldsymbol I}:|I|=k}A_I)$) except when \eqref{eqn_soln} equals $1$. This is because, when \eqref{eqn_soln} does not equal $1$, there is some Type I distribution that attains this bound, and only stacks with exactly $k$ overlapping layers contribute to the upper bound. However, when \eqref{eqn_soln} equals $1$ (so $r^\ast=0$), a Type II distribution attains this bound, where the probability for stacks of more than $k$ layers is as small as possible (intuitively, these stacks are as ``thin'' and as ``tall'' as possible). In this case, each stack with $k$ layers for ``$A$'' events is the complement of a stack with $(n-k)$ layers for ``$A^c$'' events, and Type II distributions in terms of ``$A$'' events are Type I distributions in terms of ``$A^c$'' events. Consequently, the upper bound for ${\mathbb{P}}(k\text{--out--of--}n\text{ events})$ is indirectly given by these complementary Type I distributions (e.g. Fig.~\ref{fig_complementarystack_upperbound}). In summary, the upper bound for ${\mathbb{P}}(k\text{--out--of--}n\text{ events})$ is \eqref{eqn_soln}, except when \eqref{eqn_soln} is $1$; then, the upper bound is obtained from the complementary distribution as the upper bound for ${\mathbb{P}}(\text{at least }(n-k)\text{--out--of--}n\text{ events do not occur})$, by substituting each $p_i$ in \eqref{eqn_soln} with $(1-p_{n-i+1})$ and substituting $k$ with $(n-k)$. The bound is $\min\{\sum_{i=r^\ast+1}^{n}(1-p_i)/(n-k-r^\ast),1\}$ for $r^\ast$ many dominating ``$A^c$''s (see \eqref{eqn_soln_sup_2})\footnote{In general, a distribution's $r^\ast$ differs from that of its complementary distribution.}.

\begin{figure}[t!]
    \centering
	\begin{tikzpicture}
            \begin{scope}
                \coordinate (r1lu) at (0.3,-1.85);
                \coordinate (r1ld) at (0.3,-1.75);
                \coordinate (r1rd) at (3.7,-1.75);
                \coordinate (r1ru) at (3.7,-1.85);
                \fill[fill=gray!70] (r1lu) -- (r1ld) -- (r1rd) -- (r1ru) -- cycle;

                \coordinate (r2lu) at (0.3,-0.85);
                \coordinate (r2ld) at (0.3,-0.75);
                \coordinate (r2rd) at (3.7,-0.75);
                \coordinate (r2ru) at (3.7,-0.85);
                \fill[fill=gray!70] (r2lu) -- (r2ld) -- (r2rd) -- (r2ru) -- cycle;

                \coordinate (r3lu) at (0.3,-1.35);
                \coordinate (r3ld) at (0.3,-1.25);
                \coordinate (r3rd) at (3.7,-1.25);
                \coordinate (r3ru) at (3.7,-1.35);
                \fill[fill=gray!70] (r3lu) -- (r3ld) -- (r3rd) -- (r3ru) -- cycle;

                \coordinate (r3lu) at (-0.85,-1.35);
                \coordinate (r3ld) at (-0.85,-1.25);
                \coordinate (r3rd) at (0.2,-1.25);
                \coordinate (r3ru) at (0.2,-1.35);
                \fill[fill=gray!70] (r3lu) -- (r3ld) -- (r3rd) -- (r3ru) -- cycle;

                \coordinate (r3lu) at (-0.85,-1.85);
                \coordinate (r3ld) at (-0.85,-1.75);
                \coordinate (r3rd) at (0.2,-1.75);
                \coordinate (r3ru) at (0.2,-1.85);
                \fill[fill=gray!70] (r3lu) -- (r3ld) -- (r3rd) -- (r3ru) -- cycle;

                \coordinate (r3lu) at (-0.85,-0.85);
                \coordinate (r3ld) at (-0.85,-0.75);
                \coordinate (r3rd) at (0.2,-0.75);
                \coordinate (r3ru) at (0.2,-0.85);
                \fill[fill=gray!70] (r3lu) -- (r3ld) -- (r3rd) -- (r3ru) -- cycle;

                \coordinate (r3lu) at (3.8,-1.35);
                \coordinate (r3ld) at (3.8,-1.25);
                \coordinate (r3rd) at (4.85,-1.25);
                \coordinate (r3ru) at (4.85,-1.35);
                \fill[fill=gray!70] (r3lu) -- (r3ld) -- (r3rd) -- (r3ru) -- cycle;

                \coordinate (r3lu) at (3.8,-1.85);
                \coordinate (r3ld) at (3.8,-1.75);
                \coordinate (r3rd) at (4.85,-1.75);
                \coordinate (r3ru) at (4.85,-1.85);
                \fill[fill=gray!70] (r3lu) -- (r3ld) -- (r3rd) -- (r3ru) -- cycle;

                \coordinate (r3lu) at (3.8,-0.85);
                \coordinate (r3ld) at (3.8,-0.75);
                \coordinate (r3rd) at (4.85,-0.75);
                \coordinate (r3ru) at (4.85,-0.85);
                \fill[fill=gray!70] (r3lu) -- (r3ld) -- (r3rd) -- (r3ru) -- cycle;
                
                \coordinate (r3lu) at (0.3,-1.35);
                \coordinate (r3ld) at (0.3,-1.25);
                \coordinate (r3rd) at (1.4,-1.25);
                \coordinate (r3ru) at (1.4,-1.35);
                \fill[fill=black] (r3lu) -- (r3ld) -- (r3rd) -- (r3ru) -- cycle;

                \coordinate (r2lu) at (1.4,-0.85);
                \coordinate (r2ld) at (1.4,-0.75);
                \coordinate (r2rd) at (2.5,-0.75);
                \coordinate (r2ru) at (2.5,-0.85);
                \fill[fill=black] (r2lu) -- (r2ld) -- (r2rd) -- (r2ru) -- cycle;

                \coordinate (r1lu) at (2.5,-1.85);
                \coordinate (r1ld) at (2.5,-1.75);
                \coordinate (r1rd) at (3.7,-1.75);
                \coordinate (r1ru) at (3.7,-1.85);
                \fill[fill=black] (r1lu) -- (r1ld) -- (r1rd) -- (r1ru) -- cycle;

                \draw[line width=0.07cm,white] (1.4,-0.7) -- (1.4,-2.5);
                
                \draw[line width=0.07cm,white] (2.5,-0.7) -- (2.5,-2.5);





                    
                \draw[-{Latex[length=3mm]},line width=0.06cm] (-1,-2.25) -- (5.85,-2.25); 
                \node[anchor=north,scale=1] (x_axis_label) at (5,-2.65) {$\pmb{1}$}; 

                \node[anchor=north east,scale=1] (x_axis_label) at (-0.55,-2.62) {$\pmb{0}$}; 

                \draw[-{Latex[length=3mm]},line width=0.06cm] (-1,-2.25) -- (5.85,-2.25); 

                \coordinate (p101_l) at (1.45,-3);
                \coordinate (p101_r) at (2.55,-3);
                \draw[|-|,line width=0.06cm] ($(p101_l)!5mm!90:(p101_r)$)--($(p101_r)!5mm!-90:(p101_l)$);
                \node[anchor=west,scale=1] (p101_lab) at (1.6,-2.9) {$p_{101}$}; 

                \coordinate (p110_l) at (0.3,-3);
                \coordinate (p110_r) at (1.4,-3);
                \draw[|-|,line width=0.06cm] ($(p110_l)!5mm!90:(p110_r)$)--($(p110_r)!5mm!-90:(p110_l)$);
                \node[anchor=west,scale=1] (p110_lab) at (0.45,-2.9) {$p_{110}$}; 

                \coordinate (p011_l) at (2.6,-3);
                \coordinate (p011_r) at (3.8,-3);
                \draw[|-|,line width=0.06cm] ($(p011_l)!5mm!90:(p011_r)$)--($(p011_r)!5mm!-90:(p011_l)$);
                \node[anchor=west,scale=1] (p011_lab) at (2.8,-2.9) {$p_{011}$}; 

                \coordinate (p000_l) at (-0.85,-3);
                \coordinate (p000_r) at (0.26,-3);
                \draw[|-|,line width=0.06cm] ($(p000_l)!5mm!90:(p000_r)$)--($(p000_r)!5mm!-90:(p000_l)$);

                \coordinate (p000_l) at (3.85,-3);
                \coordinate (p000_r) at (5,-3);
                \draw[|-|,line width=0.06cm] ($(p000_l)!5mm!90:(p000_r)$)--($(p000_r)!5mm!-90:(p000_l)$);

                 \node[anchor=west,scale=1] (G6) at (1.6,-3.5) {$p_{111}$};
                \node[scale=1.5] (G7) at (-0.15,-3.3) {$\pmb{\underbrace{}{}}$};    
                \draw[line width = {0.05cm}] (-0.2,-3.5) .. controls (-0.2,-3.9) .. (G6);
                \node[scale=1.4] (G8) at (4.5,-3) {$\pmb{\underbrace{}{}}$};
                \draw[line width = {0.05cm}] (4.45,-3.2) .. controls (4.5,-3.9) .. (G6);

                
                \node[anchor=west,scale=1] (p1_lab) at (-1.5,-1.8) {$p_{1}$}; 
                \node[anchor=west,scale=1] (p2_lab) at (-1.5,-0.8) {$p_{2}$}; 
                \node[anchor=west,scale=1] (p3_lab) at (-1.5,-1.3) {$p_{3}$}; 

            \end{scope}    
        \end{tikzpicture}
    \caption{\textit{A Type II distribution (gray stacks) that attains the upper bound value $1$ for ${\mathbb{P}}(\text{at least }2\text{--out--of--}3\text{ events})$ (cf. Fig.\ \ref{fig_typeIIdistribution}). Here, $r^\ast=0$. Its complementary Type I distribution has 3 stacks with $1$ (black) layer each. While not all of the gray stacks have exactly $2$ layers (hard to compute ${\mathbb{P}}(2\text{--out--of--}3\text{ events})$), all the black stacks have $1$ layer (easy to compute the equivalent ${\mathbb{P}}(1\text{--out--of--}3\text{ events do not occur})$, as $\sum_{i=1}^{3}(1-p_i)$).}
    }
	\label{fig_complementarystack_upperbound}
\end{figure}
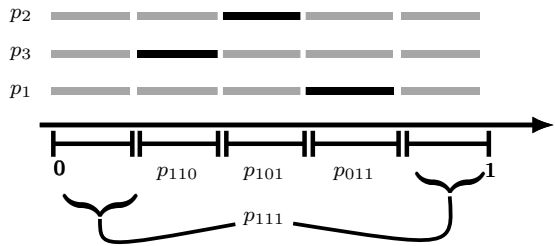

\begin{figure}[t!]
    \centering
	\begin{tikzpicture}
            \begin{scope}
                \coordinate (r1lu) at (0.3,-1.85);
                \coordinate (r1ld) at (0.3,-1.75);
                \coordinate (r1rd) at (3.7,-1.75);
                \coordinate (r1ru) at (3.7,-1.85);
                \fill[fill=black] (r1lu) -- (r1ld) -- (r1rd) -- (r1ru) -- cycle;

                \coordinate (r2lu) at (0.3,-0.85);
                \coordinate (r2ld) at (0.3,-0.75);
                \coordinate (r2rd) at (3.7,-0.75);
                \coordinate (r2ru) at (3.7,-0.85);
                \fill[fill=black] (r2lu) -- (r2ld) -- (r2rd) -- (r2ru) -- cycle;

                \coordinate (r3lu) at (0.3,-1.35);
                \coordinate (r3ld) at (0.3,-1.25);
                \coordinate (r3rd) at (3.7,-1.25);
                \coordinate (r3ru) at (3.7,-1.35);
                \fill[fill=gray!70] (r3lu) -- (r3ld) -- (r3rd) -- (r3ru) -- cycle;

                \coordinate (r3lu) at (-0.85,-1.35);
                \coordinate (r3ld) at (-0.85,-1.25);
                \coordinate (r3rd) at (0.2,-1.25);
                \coordinate (r3ru) at (0.2,-1.35);
                \fill[fill=black] (r3lu) -- (r3ld) -- (r3rd) -- (r3ru) -- cycle;

                \coordinate (r3lu) at (-0.85,-1.85);
                \coordinate (r3ld) at (-0.85,-1.75);
                \coordinate (r3rd) at (0.2,-1.75);
                \coordinate (r3ru) at (0.2,-1.85);
                \fill[fill=black] (r3lu) -- (r3ld) -- (r3rd) -- (r3ru) -- cycle;

                \coordinate (r3lu) at (-0.85,-0.85);
                \coordinate (r3ld) at (-0.85,-0.75);
                \coordinate (r3rd) at (0.2,-0.75);
                \coordinate (r3ru) at (0.2,-0.85);
                \fill[fill=gray!70] (r3lu) -- (r3ld) -- (r3rd) -- (r3ru) -- cycle;

                \coordinate (r3lu) at (3.8,-1.35);
                \coordinate (r3ld) at (3.8,-1.25);
                \coordinate (r3rd) at (4.85,-1.25);
                \coordinate (r3ru) at (4.85,-1.35);
                \fill[fill=gray!70] (r3lu) -- (r3ld) -- (r3rd) -- (r3ru) -- cycle;
                
                \coordinate (r3lu) at (2.5,-0.86);
                \coordinate (r3ld) at (2.5,-0.745);
                \coordinate (r3rd) at (3.7,-0.745);
                \coordinate (r3ru) at (3.7,-0.86);
                \fill[fill=gray!70] (r3lu) -- (r3ld) -- (r3rd) -- (r3ru) -- cycle;

                \coordinate (r3lu) at (3.8,-1.85);
                \coordinate (r3ld) at (3.8,-1.75);
                \coordinate (r3rd) at (4.85,-1.75);
                \coordinate (r3ru) at (4.85,-1.85);
                \fill[fill=gray!70] (r3lu) -- (r3ld) -- (r3rd) -- (r3ru) -- cycle;

                \coordinate (r3lu) at (3.8,-0.85);
                \coordinate (r3ld) at (3.8,-0.75);
                \coordinate (r3rd) at (4.85,-0.75);
                \coordinate (r3ru) at (4.85,-0.85);
                \fill[fill=gray!70] (r3lu) -- (r3ld) -- (r3rd) -- (r3ru) -- cycle;
                



                
                \draw[line width=0.07cm,white] (2.5,-0.7) -- (2.5,-2.5);





                    
                \draw[-{Latex[length=3mm]},line width=0.06cm] (-1,-2.25) -- (5.85,-2.25); 
                \node[anchor=north,scale=1] (x_axis_label) at (5,-2.65) {$\pmb{1}$}; 

                \node[anchor=north east,scale=1] (x_axis_label) at (-0.55,-2.62) {$\pmb{0}$}; 

                \draw[-{Latex[length=3mm]},line width=0.06cm] (-1,-2.25) -- (5.85,-2.25); 

                \coordinate (p001_l) at (0.3,-3);
                \coordinate (p001_r) at (2.55,-3);
                \draw[|-|,line width=0.06cm] ($(p001_l)!5mm!90:(p001_r)$)--($(p001_r)!5mm!-90:(p001_l)$);
                \node[anchor=west,scale=1] (p001_lab) at (1.1,-2.9) {$p_{001}$}; 

                \node[anchor=west,scale=1] (p001_lab) at (-0.65,-2.9) {$p_{010}$}; 

                \coordinate (p011_l) at (2.6,-3);
                \coordinate (p011_r) at (3.8,-3);
                \draw[|-|,line width=0.06cm] ($(p011_l)!5mm!90:(p011_r)$)--($(p011_r)!5mm!-90:(p011_l)$);
                \node[anchor=west,scale=1] (p011_lab) at (2.8,-2.9) {$p_{011}$}; 

                \coordinate (p000_l) at (-0.85,-3);
                \coordinate (p000_r) at (0.26,-3);
                \draw[|-|,line width=0.06cm] ($(p000_l)!5mm!90:(p000_r)$)--($(p000_r)!5mm!-90:(p000_l)$);
                \node[anchor=west,scale=1] (p000_lab) at (4,-2.9) {$p_{111}$}; 

                \coordinate (p000_l) at (3.85,-3);
                \coordinate (p000_r) at (5,-3);
                \draw[|-|,line width=0.06cm] ($(p000_l)!5mm!90:(p000_r)$)--($(p000_r)!5mm!-90:(p000_l)$);


                
                \node[anchor=west,scale=1] (p1_lab) at (-1.5,-1.8) {$p_{1}$}; 
                \node[anchor=west,scale=1] (p2_lab) at (-1.5,-0.8) {$p_{2}$}; 
                \node[anchor=west,scale=1] (p3_lab) at (-1.5,-1.3) {$p_{3}$}; 

            \end{scope}    
        \end{tikzpicture}
    \caption{\textit{A Type I distribution (with black stacks) that attains the upper bound $\sum_{i=2}^{3}(1-p_i)$ for ${\mathbb{P}}(\text{at least }2\text{--out--of--}3\text{ events do not occur})$ (cf. Fig.\ \ref{fig_typeIdistribution}). Here, $r^\ast=1$, since the $A_1^c$ layer dominates. There is a complementary, $2$-layer, gray stack on top of the single-layer, black stack for the probability $p_{011}$. This is the only stack with $2$--out--of--$3$ ``$A$''  events, so the lower bound for ${\mathbb{P}}(2\text{--out--of--}3\text{ events})$ is $p_{011}$. Computing $p_{011}$ is difficult using the gray stacks, but easy using the black stacks; focusing on the dominating $A_1^c$ layer, $p_{011}$ is the part of $(1-p_1)$ that does not contribute to the ``at least $2$--out--of--$3$ events do not occur'' probability, so $p_{011}=(1-p_1)-\sum_{i=2}^{3}(1-p_i)$.}
    }
	\label{fig_comlementarystack_lowerbound}
\end{figure}
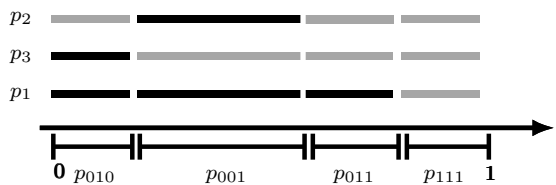  
Obtaining lower bounds for ${\mathbb{P}}(k\text{--out--of--}n\text{ events})$ is slightly more involved. \footnote{By the monotonicity of measures, when \eqref{eqn_soln_inf_2} is $0$, the lower bound for $\mathbb{P}(k\text{--out--of--}n\text{ events})$ is also $0$.}When \eqref{eqn_soln_inf_2} is not $0$, \eqref{eqn_soln_inf_2} is indirectly attained by a Type I distribution of ``$A^c$'' events, where only stacks with $(n-k+1)$ layers contribute to the lower bound. Here, use
\begin{align}
\label{eqn_exactlykoutofn_complmntryProb}
&{\mathbb{P}}(\text{at least }k\text{--out--of--}n\text{ events})\nonumber\\ =& 1-{\mathbb{P}}(\text{at most }(k-1)\text{--out--of--}n\text{ events})\nonumber\\ =&1-{\mathbb{P}}(\text{at least }(n-k+1)\text{--out--of--}n\text{ events do not occur})\,. \end{align} Such distributions have $r^\ast$ many ``$A^c$'' events that dominate the other ``$A^c$'' events; i.e. these events are part of all compound events that satisfy ``exactly ($n-k+1$) events do not occur''. This means only stacks consisting of these dominating events can define complementary stacks with exactly $k$ many ``$A$'' events. Indeed, it is \emph{only} when $r^\ast=n-k$ in \eqref{eqn_soln_inf_2} that we \emph{can} have such a stack with exactly $k$ many ``$A$'' events. This complementary stack represents a probability that is the length of the smallest overlap of all leftover layers from the dominating ``$A^c$'' events (e.g. Fig.~\ref{fig_comlementarystack_lowerbound}). So, for $r^\ast=n-k$, the lower bound for ${\mathbb{P}}(k\text{--out--of--}n\text{ events})$ is 
\begin{align*}
\left\{\begin{array}{ll}
\max\{\sum_{i=1}^{n-k}(1-p_i)-\sum_{i=n-k+1}^{n}(1-p_i) -(n-k-1),\,0\}\,;&\text{if }n>k \\
\max\{1-\sum_{i=1}^{n}(1-p_i),\,0\}\hfill;&\text{if }n=k
\end{array}\right.
\end{align*}
If instead, $r^\ast< n-k$ in \eqref{eqn_soln_inf_2}, the dominating layers cannot form a complementary stack with exactly $k$ layers, so the lower bound for ${\mathbb{P}}(k\text{--out--of--}n\text{ events})$ is $0$.  

\paragraph{The significance of $\pmb{r^\ast}$.} For distributions that attain the upper bound \eqref{eqn_soln}, $r^\ast$ is the number of ``$A$'' events that always occur whenever at least $k$ many ``$A$''s occur. 
$r^\ast$ also defines a stopping rule in a targeted search for a bound. With a sorted list of the ``$p_i$''s, $r^\ast=0$ in the special cases when either $k=1$ (for upper bounds) or $k=n$ (for lower bounds). When not in these special cases, one could begin the search by checking if the inequality defining $r^\ast$ (in \eqref{eqn_soln} or \eqref{eqn_soln_inf_2}) is satisfied when $r=1$. If it is, then there is a search for the maximum $r$ that satisfies this inequality. Otherwise, if $r=1$ does not satisfy the inequality, then $r^\ast=0$. An efficient search for $r^\ast$ could use the following fact recursively: if the inequality is satisfied for some $r>1$, then it is satisfied for all smaller $r$ up to $r=1$ (e.g. see step 4 of Theorem \ref{thrm_sup}'s proof). There are, at most, $(k-1)$ inequality checks for $r^\ast$ in a search for an upper bound, and at most $(n-k)$ in a search for a lower bound. These ``number of checks'' can be considerably reduced (e.g. by using a modified binary search) to $\mathcal{O}(\log k)$ for upper bound searches when $k>1$, and $\mathcal{O}(\log (n-k))$ for lower bound searches when $n>k$. These savings can be significant if searching for several ``$k$'' (fixed $n$) bounds. 

The DP formulation in Section~\ref{subsec_DPformulation} is
primarily a structural tool: it identifies the Boole--Fr\'{e}chet bound $\Phi^\ast$ from an LP problem with the value of an
equivalent control problem, and characterizes  the value function as a bounded fixed
point of the Bellman operator $\mathcal T$.  From an algorithmic perspective,
the statewise linear programme of Lemma~\ref{lem:statewise-Phi} already shows
that, in general, computing $\Phi^\ast = \Phi(1,{\boldsymbol p},{\boldsymbol 0})$ exactly requires solving
an LP with $|\boldsymbol I| = 2^n$ variables and $n$ (or more) linear
constraints; exponential in the
number $n$ of atomic events.  Thus, the DP framework of
Theorem~\ref{thm:bellmanconvergence} should
be viewed as a constructive representation of the sharp bounds, with genuine
computational gains arising only in special cases where the extremal structure
simplifies.  In particular, in the ``$k$--out--of--$n$'' setting, the resulting
sharp bounds can be evaluated in $\mathcal O(n)$ time for sorted $p_i$, or $\mathcal O(n\log n)$ for unsorted $p_i$,
using the stopping index $r^\ast$. When multiple bounds are computed, each subsequent bound is either $\mathcal O(\log k)$ or $\mathcal O(\log (n-k))$.

A natural generalization of $r^*$ presents itself within the DP formulation. For a non-trivial compound event $E$, let set $\mathcal Y(E,\boldsymbol p)$ contain all optimal configurations of stacks that achieve the upper bound on $\mathbb P(E)$ subject to marginals $\boldsymbol p$; i.e. for ${\boldsymbol y}^*\in\mathcal Y(E,\boldsymbol p)$, $e({\boldsymbol y}^*)=\sum_{I\in E}y^*_I=\Phi^*=\sup\{\;\mathbb P(E)\mid\mathbb P\text{ has marginals }{\boldsymbol p}\}$ where $M{\boldsymbol y}^*={\boldsymbol p}$. The index set of $s$-dominating $A_i$ layers in configuration $y^*$ is $D^s({\boldsymbol y}^*):=\{\;i: |\{\;I\in E\mid y^*_I>0,\, I_i=0\}|\leqslant s\}$~---~these layers are in all but, at most, $s$ of the stacks that contribute to $\mathbb P(E)$. There are hierarchies of $r^*$ generalizations, such as $r_{E,s}^{\max}({\boldsymbol p}):=\max_{{\boldsymbol y}^*\in\mathcal Y(E,\boldsymbol p)}|D^s({\boldsymbol y}^*)|$ and $r_{E,s}^{\min}({\boldsymbol p}):=\min_{{\boldsymbol y}^*\in\mathcal Y(E,\boldsymbol p)}|D^s({\boldsymbol y}^*)|$. In particular, the largest number of dominating $A_i$ layers at optimality (i.e. layers present in all stacks that contribute to $\mathbb P(E)$) is $r_{E,0}^{\max}({\boldsymbol p})$. By definition, $D^0({\boldsymbol y}^*)\subseteq D^1({\boldsymbol y}^*)\subseteq\ldots $ and $r_{E,0}^{\max}({\boldsymbol p})\leqslant r_{E,1}^{\max}({\boldsymbol p})\leqslant\ldots$, which suggests a hierarchy of optimal configuration/distribution ``Types'' classified by $s$.

\paragraph{\textbf{More examples.}}
\label{sec_more_examples}
We have focused on bounding ``out--of--$n$'' probabilities. Optimal policies of Theorem\ \ref{thm:bellmanconvergence} give constructions of bounds on more general compound event probabilities, as illustrated by the following example. 
Suppose there are $k$ sets of distinct events, where the $i$-th set contains $n_i$ many events. $A_{ij}$ is the $j$-th event in the $i$-th set, and  
$0<p_{ij}<1$ is its probability. Further assume, \emph{w.l.o.g.}, that $p_{i1}<p_{i2}<\ldots<p_{i,n_i}$ for each set $i$, and that $p_{1 n_1}>p_{2 n_2}>\ldots>p_{k n_k}$. What is the largest probability of all the ``$A$'' events from one set, and none of the $A$ events from the other $k-1$ sets, occurring?   
\begin{theorem}
\label{thrm_nonoverlapstacksproblem}
With the $A_{ij}$ events and $p_{ij}$ probabilities as defined,
\begin{align}
\label{eqn_nonoverlapstacksproblem}
&\sup_{\mathbb{P}}\mathbb{P}(\text{all the ``}A\text{'' events from one set, and none from the other sets, occur})\nonumber \\
\text{s.t.}\,&\,\,\mathbb{P}(A_{ij})=p_{ij}\ \text{ for all }i=1,...,k \text{ and }j=1,...,n_i
\end{align}
has the solution
\begin{align}
\label{eqn_soln_nonoverlapstacksproblem}
\sum_{i=1}^{r_{t\ast}} \min\{p_{in_i} - t^\ast,\,p_{i1}\}
\end{align}
where 

$$t^\ast:=\min\!\left\{\,0\leqslant t\leqslant p_{1n_1}\,\,\bigg\vert\,\,\sum_{i=1}^{r_t} \min\{p_{in_i} - t, p_{i1}\}\leqslant 1-\max\{t,\, p^\ast_{1},\ldots,p^\ast_{r_t}\}\right\}\!,$$
with $\,\,r_{t}:=\max\!\big\{\,r\in\{1,\ldots, k\}\,\,\big\vert\,\, p_{r n_r}>t\big\}\,\,$ and $\,\,\ p^\ast_i:=p_{i n_i}-p_{i1}$.
\end{theorem}

\begin{proof}
Existence: There exist distributions that satisfy the constraints in \eqref{eqn_nonoverlapstacksproblem} (e.g. stack all of the ``$A$'' events' layers over the unit interval). Suppose $\mathbb{P}$ satisfies the constraints. Consider the stack that represents the probability of only all of the events in set $i$ occurring, $\mathbb{P}(\cap_{j=1}^{n_i}A_{ij}\cap_{l\ne i}^k\cap_{m=1}^{n_l}A^c_{lm})$. The relevant ``$A$'' event layers that contribute to this stack have lengths between $p_{i1}$ and $p_{in_i}$ --- this means, in order to maximize this stack's horizontal length, one can arrange these layers so that they dominate the ``$A_{i1}$'' layer, and they are dominated by the ``$A_{in_i}$'' layer. Consequently, \emph{w.l.o.g.}, there is no need to explicitly work with any layers in this stack except the ``$A_{i1}$'' and ``$A_{i n_i}$'' layers.

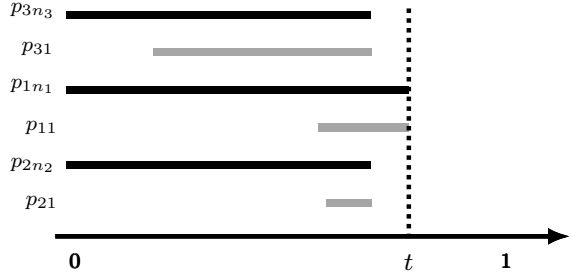
\begin{figure}[!t]
    \centering
	\begin{tikzpicture}
            \begin{scope}
                \coordinate (r1lu) at (2.6,-1.85);
                \coordinate (r1ld) at (2.6,-1.75);
                \coordinate (r1rd) at (3.2,-1.75);
                \coordinate (r1ru) at (3.2,-1.85);
                \fill[fill=gray!70] (r1lu) -- (r1ld) -- (r1rd) -- (r1ru) -- cycle;

                \coordinate (r4lu) at (-0.85,-0.25);
                \coordinate (r4ld) at (-0.85,-0.35);
                \coordinate (r4rd) at (3.7,-0.35);
                \coordinate (r4ru) at (3.7,-0.25);
                \fill[fill=black] (r4lu) -- (r4ld) -- (r4rd) -- (r4ru) -- cycle;

                \coordinate (r6lu) at (0.3,0.25);
                \coordinate (r6ld) at (0.3,0.15);
                \coordinate (r6rd) at (3.2,0.15);
                \coordinate (r6ru) at (3.2,0.25);
                \fill[fill=gray!70] (r6lu) -- (r6ld) -- (r6rd) -- (r6ru) -- cycle;

                \coordinate (r5lu) at (-0.85,0.75);
                \coordinate (r5ld) at (-0.85,0.65);
                \coordinate (r5rd) at (3.2,0.65);
                \coordinate (r5ru) at (3.2,0.75);
                \fill[fill=black] (r5lu) -- (r5ld) -- (r5rd) -- (r5ru) -- cycle;


                \coordinate (r3lu) at (-0.85,-1.35);
                \coordinate (r3ld) at (-0.85,-1.25);
                \coordinate (r3rd) at (3.2,-1.25);
                \coordinate (r3ru) at (3.2,-1.35);
                \fill[fill=black] (r3lu) -- (r3ld) -- (r3rd) -- (r3ru) -- cycle;




                
                \coordinate (r3lu) at (2.5,-0.86);
                \coordinate (r3ld) at (2.5,-0.745);
                \coordinate (r3rd) at (3.7,-0.745);
                \coordinate (r3ru) at (3.7,-0.86);
                \fill[fill=gray!70] (r3lu) -- (r3ld) -- (r3rd) -- (r3ru) -- cycle;


                



                





                    
                \draw[-{Latex[length=3mm]},line width=0.06cm] (-1,-2.25) -- (5.85,-2.25); 
                \node[anchor=north,scale=1] (x_axis_label) at (5,-2.37) {$\pmb{1}$}; 

                \node[anchor=north east,scale=1] (x_axis_label) at (-0.55,-2.37) {$\pmb{0}$}; 

                \draw[-{Latex[length=3mm]},line width=0.06cm] (-1,-2.25) -- (5.85,-2.25); 

                \draw[-,line width=0.06cm,dotted] (3.7,-2.275)--(3.7,0.8);
                \node[scale=1.2] (p001_lab) at (3.7,-2.6) {$t$}; 






                
                \node[anchor=west,scale=1] (p21_lab) at (-1.5,-1.8) {$p_{21}$}; 
                \node[anchor=west,scale=1] (p1n1_lab) at (-1.7,-0.25) {$p_{1n_1}$}; 
                \node[anchor=west,scale=1] (p11_lab) at (-1.5,-0.8) {$p_{11}$}; 
                \node[anchor=west,scale=1] (p2n2_lab) at (-1.7,-1.3) {$p_{2n_2}$}; 
                \node[anchor=west,scale=1] (p31_lab) at (-1.55,0.25) {$p_{31}$}; 
                \node[anchor=west,scale=1] (p3n3_lab) at (-1.7,0.75) {$p_{3n_3}$}; 

            \end{scope}    
        \end{tikzpicture}
    \caption{\textit{Only the ``$A_{i1}$'' (gray) and ``$A_{in_i}$'' (black) layers are represented. Initially, the vertical $t$-line is located at $t=p_{1n_1}$. As illustrated, $p^\ast_{i}$ is the length of the black ``$A_{in_i}$'' layer that does not overlap with the gray ``$A_{i1}$'' layer.}}
	\label{fig_supexclusivecmpndevents_step1}
\end{figure}
\begin{figure}[!t]
    \centering
	\begin{tikzpicture}
            \begin{scope}
                \coordinate (r1lu) at (4.3,-1.85);
                \coordinate (r1ld) at (4.3,-1.75);
                \coordinate (r1rd) at (4.9,-1.75);
                \coordinate (r1ru) at (4.9,-1.85);
                \fill[fill=gray!70] (r1lu) -- (r1ld) -- (r1rd) -- (r1ru) -- cycle;

                \coordinate (r8lu) at (4.3,-1.35);
                \coordinate (r8ld) at (4.3,-1.25);
                \coordinate (r8rd) at (4.9,-1.25);
                \coordinate (r8ru) at (4.9,-1.35);
                \fill[fill=black] (r8lu) -- (r8ld) -- (r8rd) -- (r8ru) -- cycle;                

                \coordinate (r4lu) at (-0.85,-0.25);
                \coordinate (r4ld) at (-0.85,-0.35);
                \coordinate (r4rd) at (2.6,-0.35);
                \coordinate (r4ru) at (2.6,-0.25);
                \fill[fill=black] (r4lu) -- (r4ld) -- (r4rd) -- (r4ru) -- cycle;

                \coordinate (r6lu) at (0.3,0.25);
                \coordinate (r6ld) at (0.3,0.15);
                \coordinate (r6rd) at (2.6,0.15);
                \coordinate (r6ru) at (2.6,0.25);
                \fill[fill=gray!70] (r6lu) -- (r6ld) -- (r6rd) -- (r6ru) -- cycle;         

                \coordinate (r9lu) at (3.7,0.75);
                \coordinate (r9ld) at (3.7,0.65);
                \coordinate (r9rd) at (4.3,0.65);
                \coordinate (r9ru) at (4.3,0.75);
                \fill[fill=black] (r9lu) -- (r9ld) -- (r9rd) -- (r9ru) -- cycle;
                
                \coordinate (r7lu) at (3.7,0.25);
                \coordinate (r7ld) at (3.7,0.15);
                \coordinate (r7rd) at (4.3,0.15);
                \coordinate (r7ru) at (4.3,0.25);
                \fill[fill=gray!70] (r7lu) -- (r7ld) -- (r7rd) -- (r7ru) -- cycle;

                \coordinate (r5lu) at (-0.85,0.75);
                \coordinate (r5ld) at (-0.85,0.65);
                \coordinate (r5rd) at (2.6,0.65);
                \coordinate (r5ru) at (2.6,0.75);
                \fill[fill=black] (r5lu) -- (r5ld) -- (r5rd) -- (r5ru) -- cycle;


                \coordinate (r3lu) at (-0.85,-1.35);
                \coordinate (r3ld) at (-0.85,-1.25);
                \coordinate (r3rd) at (2.6,-1.25);
                \coordinate (r3ru) at (2.6,-1.35);
                \fill[fill=black] (r3lu) -- (r3ld) -- (r3rd) -- (r3ru) -- cycle;




                
                \coordinate (r3lu) at (2.6,-0.86);
                \coordinate (r3ld) at (2.6,-0.745);
                \coordinate (r3rd) at (3.7,-0.745);
                \coordinate (r3ru) at (3.7,-0.86);
                \fill[fill=gray!70] (r3lu) -- (r3ld) -- (r3rd) -- (r3ru) -- cycle;

                \coordinate (r4lu) at (2.6,-0.25);
                \coordinate (r4ld) at (2.6,-0.35);
                \coordinate (r4rd) at (3.7,-0.35);
                \coordinate (r4ru) at (3.7,-0.25);
                \fill[fill=black] (r4lu) -- (r4ld) -- (r4rd) -- (r4ru) -- cycle;                


                



                





                    
                \draw[-{Latex[length=3mm]},line width=0.06cm] (-1,-2.25) -- (5.85,-2.25); 
                \node[anchor=north,scale=1] (x_axis_label) at (5,-2.37) {$\pmb{1}$}; 

                \node[anchor=north east,scale=1] (x_axis_label) at (-0.55,-2.37) {$\pmb{0}$}; 

                \draw[-{Latex[length=3mm]},line width=0.06cm] (-1,-2.25) -- (5.85,-2.25); 

                \draw[-,line width=0.06cm,dotted] (2.6,-2.275)--(2.6,0.8);
                \node[scale=1.2] (p001_lab) at (2.6,-2.6) {$t^\ast$}; 






                
                \node[anchor=west,scale=1] (p21_lab) at (-1.5,-1.8) {$p_{21}$}; 
                \node[anchor=west,scale=1] (p1n1_lab) at (-1.7,-0.25) {$p_{1n_1}$}; 
                \node[anchor=west,scale=1] (p11_lab) at (-1.5,-0.8) {$p_{11}$}; 
                \node[anchor=west,scale=1] (p2n2_lab) at (-1.7,-1.3) {$p_{2n_2}$}; 
                \node[anchor=west,scale=1] (p31_lab) at (-1.55,0.25) {$p_{31}$}; 
                \node[anchor=west,scale=1] (p3n3_lab) at (-1.7,0.75) {$p_{3n_3}$}; 

            \end{scope}    
        \end{tikzpicture}
    \caption{\textit{Eventually, $t$ cannot be reduced further: here, stacks to the right of $t^\ast$ sum up to $1-t^\ast$, and the overlapping of all black stacks to the left of $t^\ast$ preclude any stacks to the left of $t^\ast$ that contribute to the objective function. This illustration also suggests $t^\ast=p^\ast_1=p^\ast_2$ with $r_{t^\ast}=3$. Since $t^\ast> p^\ast_3$, the solution to \eqref{eqn_nonoverlapstacksproblem} is $\sum_{i=1}^{3} (p_{in_i}-t^\ast)=1-t^\ast=1-p^\ast_1=1-p^\ast_2$. Distributions that attain such bounds are unique up to reshuffling the stack locations.}}
	\label{fig_supexclusivecmpndevents_step2}
\end{figure}
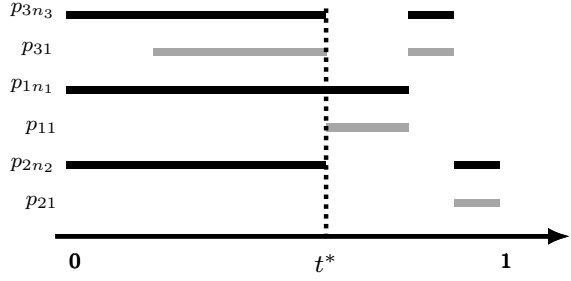
Starting from any feasible distribution $\mathbb{P}$ that satisfies the constraints in \eqref{eqn_nonoverlapstacksproblem}, we may rearrange only those portions of layers that do not contribute to the objective function. Such rearrangements preserve the constraints and cannot reduce the objective function value (and if the objective value increases, we simply take the new $\mathbb{P}$ as our starting point). By repeatedly applying these rearrangements we obtain a feasible distribution $\tilde{\mathbb{P}}$, with objective value at least that of $\mathbb{P}$, in which all portions of layers that do not contribute to the objective function overlap as much as possible, and all portions that do contribute belong to stacks that do not overlap with other stacks. Thus, \emph{w.l.o.g.}, we may restrict attention to configurations of this canonical form. This suggests constructing a local maximum as follows. First, overlap all layers, so that ``$A_{in_i}$'' layers are aligned with $0$ on the left, and ``$A_{i1}$'' layers are aligned with their respective ``$A_{in_i}$'' layers on the right (e.g. Fig.\ \ref{fig_supexclusivecmpndevents_step1}). Define a vertical line (representing the $t$ variable) that is initially located at $t=p_{1 n_1}$. Reduce $t$ continuously by moving this vertical line leftward. As $t$ decreases, rearrange the portions of ``$A_{i1}$'' and ``$A_{in_i}$'' layers that lie to the right of the $t$ line, ensuring that they form non-overlapping stacks that contribute to the objective function. Keep decreasing $t$ until either these non-overlapping stacks sum to $1-\max\{t,\, p^\ast_{1},\ldots,p^\ast_{r_t}\}$, or sum to $\sum_{i=1}^{k} p_{i1}$, whichever happens first (e.g. Fig.\ \ref{fig_supexclusivecmpndevents_step2}). 

Uniqueness: The global uniqueness of this maximum can be shown via a contradiction. \emph{W.l.o.g.}, assume $t^\ast$ has the value $t^{\ast\ast}$, and assume there is some configuration of stacks that gives a larger objective function value than that determined by $t^{**}$. Rearrange the locations of stacks in this configuration, so that all stacks that do not contribute to the objective function are located over the unit interval on the left, while the remaining stacks are located to the right of all of these non-contributing stacks. Ensure that the portions of ``$A_{in_i}$'' layers in the non-contributing stacks are aligned at $0$, and the ``$A_{i1}$'' portions in these stacks are aligned with their respective non-contributing ``$A_{in_i}$'' portions on the right\footnote{Since these are non-contributing portions, rearranging them cannot reduce the objective function value. If the objective function value is increased, then consider this new value instead.}. Among all ``$A_{in_i}$'' layers with contributing portions, there is some layer with the smallest non-contributing portion. Let the length of this smallest portion be $l$. Note that $l\ne t^{\ast\ast}$, since $t^{\ast\ast}$ determines a smaller objective function value than that given by the initial configuration that determines $l$. Set $t=l$ and reverse the process in the previous paragraph\footnote{This entails ensuring that if an ``$A_{in_i}$'' or ``$A_{i1}$'' layer touches the $t$ line from the left, it will continue to do so until there is none of the layer to the right of $t$.}. Continue to increase $t$, reassigning layers in contributing stacks to non-contributing stacks, until there are no more contributing stacks; i.e. until $t=p_{1n_1}$. By symmetry and the definition of $t^\ast$ in Theorem \ref{thrm_nonoverlapstacksproblem}, increasing $t$ from $l$ to $p_{1n_1}$ shows that reducing $t$ from $p_{1n_1}$ gives $t^\ast=l$. We appear to have contradicting values for $t^\ast$, namely $t^\ast=l\ne t^{\ast\ast}=t^\ast$.
\end{proof}

\noindent\textbf{Remarks.} Theorem\ \ref{thrm_nonoverlapstacksproblem} solves $\sup_{\mathbb{P}}\mathbb{P}(\cup_{i=1}^{k}\cap_{j=1}^{n_i}A_{ij}\cap_{l\ne i}^k\cap_{m=1}^{n_{l}} A^c_{lm})$. The analogous result for $\inf_{\mathbb{P}} \mathbb{P}(\cap_{i=1}^{k}\cup_{j=1}^{n_i}A_{ij}\cup_{l\ne i}^k\cup_{m=1}^{n_{l}} A^c_{lm})$ is derived from Theorem\ \ref{thrm_nonoverlapstacksproblem} by using complementary events and complementary probabilities. Moreover, from the symmetry between complementary stacks, one deduces $\sup_{\mathbb{P}} \mathbb{P}(\cup_{i=1}^{k}\cap_{j=1}^{n_i}A^c_{ij}\cap_{l\ne i}^k\cap_{m=1}^{n_{l}} A_{lm})$ and $\inf_{\mathbb{P}} \mathbb{P}(\cap_{i=1}^{k}\cup_{j=1}^{n_i}A^c_{ij}\cup_{l\ne i}^k\cup_{m=1}^{n_{l}} A_{lm})$.

The related, weaker bound, $$\sup_{\mathbb{P}}\mathbb{P}(\text{\emph{all the ``}}A\text{\emph{'' events from one of the }}k\text{\emph{ sets occur}})$$ (i.e. $\sup_{\mathbb{P}}\mathbb{P}(\cup_{i=1}^k\cap_{j=1}^{n_i}A_{ij})$), is equal to the Theorem\ \ref{thrm_nonoverlapstacksproblem} bound when $t^\ast=0$; then, all of the ``$A_{i1}$'' layers belong to stacks that contribute to the objective function. More generally, this weaker bound follows immediately from $\sup_{\mathbb{P}}\mathbb{P}(\text{at least }1\text{--out--of--}k)$ via Theorem\ \ref{thrm_sup}, using only the $k$ ``$p_{i1}$'' probabilities: the bound is $\min\{\sum_{i=1}^k p_{i1}, 1\}$. Intuitively, the contributing stacks are ``spread out'' as much as they can be over the unit interval without overlapping, unless overlapping portions of ``$A_{i1}$'' layers in these stacks is unavoidable.

\section*{Acknowledgments.}
This work was partly inspired by an application of Boole--Fr\'{e}chet bounds in software reliability assessment, due to  Prof. Lorenzo Strigini, Dr. Andrey Povyakalo, and Dr. David Wright. 

\bibliographystyle{vancouver}
\bibliography{refs.bib}

@article{DuranteSarkociSempi2009,
  title   = {Shuffles of copulas},
  author  = {Durante, Fabrizio and Sarkoci, Peter and Sempi, Carlo},
  journal = {Journal of Mathematical Analysis and Applications},
  year    = {2009},
  volume  = {352},
  number  = {2},
  pages   = {914--921},
  issn    = {0022-247X},
  doi     = {10.1016/j.jmaa.2008.11.064},
  url     = {https://doi.org/10.1016/j.jmaa.2008.11.064}
}

@article{MikusinskiSherwoodTaylor1992,
  title   = {Shuffles of Min},
  author  = {Mikusi{\'n}ski, Piotr and Sherwood, Howard and Taylor, Michael D.},
  journal = {Stochastica},
  year    = {1992},
  volume  = {13},
  number  = {1},
  pages   = {61--74},
  issn    = {0210-7821},
  url     = {https://eudml.org/doc/39282}
}

@book{Boole_2009, 
place={Cambridge}, 
series={Cambridge Library Collection - Mathematics}, 
title={An Investigation of the Laws of Thought: On Which Are Founded the Mathematical Theories of Logic and Probabilities},
edition={reprint},
publisher={Cambridge University Press}, 
author={Boole, George}, 
year={2009 [1854]}, 
collection={Cambridge Library Collection - Mathematics}
}

@article{Frechet_1935,
author={Fr\'{e}chet, M.},
year={1935},
title={G\'{e}n\'{e}ralisations du th\'{e}or\`{e}me des probabilit\'{e}s totales},
journal={Fundamenta Mathematicae},
volume={25},
pages={379--387}
}

@article{PuccettiRueschendorf2012,
  title   = {Computation of sharp bounds on the distribution of a function of dependent risks},
  author  = {Puccetti, Giovanni and R{\"u}schendorf, Ludger},
  journal = {Journal of Computational and Applied Mathematics},
  year    = {2012},
  volume  = {236},
  number  = {7},
  pages   = {1833--1840},
  doi     = {10.1016/j.cam.2011.10.015},
  issn    = {0377-0427},
  month   = {January},
  publisher = {Elsevier}
}

@article{Hailperin_1965,
 ISSN = {00029890, 19300972},
url={https://www.jstor.org/stable/2313491},
 doi = {10.2307/2313491},
 author = {Theodore Hailperin},
 journal = {The American Mathematical Monthly},
 number = {4},
 pages = {343--359},
 publisher = {[Taylor & Francis, Ltd., Mathematical Association of America]},
 title = {Best Possible Inequalities for the Probability of a Logical Function of Events},
 urldate = {2025-01-03},
 volume = {72},
 year = {1965}
}

@article{Bonferroni_1937,
 Author = {Bonferroni, C. E.},
Title = {Teoria statistica delle classi e calcolo delle probab\`{i}l\`{i}t\`{a}},
 Year = {1936},
pages={1-62},
number={8},
journal={{{Publicazioni} {del} {R.} {Istituto} {Superiori} di {Scienze} {Economiche} e {Commerciali} di {Firenze}}},
location={Firenze},
publisher={Seeber}
}

@article{Kounias_1976,
 ISSN = {00361399},
 URL = {http://www.jstor.org/stable/2100531},
 author = {Kounias, Stratis and Marin, Jacqueline},
 journal = {SIAM Journal on Applied Mathematics},
 number = {2},
 pages = {307--323},
 publisher = {Society for Industrial and Applied Mathematics},
 title = {Best Linear Bonferroni Bounds},
 urldate = {2025-01-03},
 volume = {30},
 year = {1976}
}

@article{Hunter_1976,
 ISSN = {00219002},
 URL = {http://www.jstor.org/stable/3212481},
 author = {David Hunter},
 journal = {Journal of Applied Probability},
 number = {3},
 pages = {597--603},
 publisher = {Applied Probability Trust},
 title = {An Upper Bound for the Probability of a Union},
 urldate = {2025-01-03},
 volume = {13},
 year = {1976}
}

@article{1967_Dawson,
author={Dawson, D. A. and Sankoff, D.},
year={1967},
title={An Inequality for Probabilities},
journal={Proceedings of the American Mathematical Society},
volume={18},
number={3},
pages={504-507},
isbn={0002-9939},
language={English},
}

@article{Kounias_1968,
author={Kounias, Eustratios G.},
year={1968},
title={Bounds for the Probability of a Union, with Applications},
journal={The Annals of mathematical statistics},
volume={39},
number={6},
pages={2154-2158},
isbn={0003-4851},
language={English},
}

@article{ChungErdos_1952,
author={Chung,K. L. and Erd\"{o}s,P.},
year={1952},
title={On the Application of the Borel-Cantelli Lemma},
journal={Transactions of the American Mathematical Society},
volume={72},
number={1},
pages={179-186},
isbn={0002-9947},
language={English},
}

@article{Gallot_1966,
author={Gallot,S.},
year={1966},
title={A bound for the maximum of a number of random variables},
journal={Journal of applied probability},
volume={3},
number={2},
pages={556-558},
isbn={0021-9002},
language={English},
}

@article{Boros_2014,
 ISSN = {0364765X, 15265471},
 URL = {http://www.jstor.org/stable/24541011},
 author = {Endre Boros and Andrea Scozzari and Fabio Tardella and Pierangela Veneziani},
 journal = {Mathematics of Operations Research},
 number = {4},
 pages = {1311--1329},
 publisher = {INFORMS},
 title = {Polynomially Computable Bounds for the Probability of the Union of Events},
 urldate = {2025-01-03},
 volume = {39},
 year = {2014}
}

@article{Bukszar_2012,
author={Bukszar,Jozsef and Madi-Nagy,Gergely and Szantai,Tamas},
year={2012},
title={Computing bounds for the probability of the union of events by different methods},
journal={Annals of operations research},
volume={201},
number={1},
pages={63-81},
isbn={0254-5330},
language={English}
}

@article{Prekopa_2005,
author={Prékopa,András and Gao,Linchun},
year={2005},
title={Bounding the probability of the union of events by aggregation and disaggregation in linear programs},
journal={Discrete Applied Mathematics},
volume={145},
number={3},
pages={444-454},
isbn={0166-218X},
language={English},
}

@book{Galambos_1996,
title={Bonferroni-type Inequalities with Applications},
author={Galambos, Janos and Simonelli, Italo},
series={Probability and Its Applications},
publisher={Springer New York, NY},
year={1996},
isbn={978-0-387-94776-1},
issn={1431-7028},
edition={1}
}

@article{manne1960,
  author  = {Manne, Alan S.},
  title   = {Linear Programming and Sequential Decisions},
  journal = {Management Science},
  volume  = {6},
  number  = {3},
  pages   = {259--267},
  year    = {1960},
  doi     = {10.1287/mnsc.6.3.259}
}

@incollection{buyuktahtakin2011,
  author    = {B{\"u}y{\"u}ktahtak{\i}n, {\.I}. Esra},
  title     = {Dynamic Programming via Linear Programming},
  booktitle = {Wiley Encyclopedia of Operations Research and Management Science},
  editor    = {Cochran, James J. and Cox, Louis A. and Keskinocak, Pinar and Kharoufeh, Jeffrey P. and Smith, J. Cole},
  publisher = {John Wiley \& Sons},
  address   = {Hoboken, NJ},
  year      = {2011},
  pages     = {1561--1566},
  doi       = {10.1002/9780470400531.eorms0277}
}

\end{document}